\documentclass[a4paper,10pt]{article}
\title{Turaev--Viro TQFT and the Rank versus Genus Conjecture}
\author{\textsc{Qing Lan}\\
}
\date{\today}

\usepackage{amsmath, amsfonts, amssymb, amsthm, amscd, graphicx, float, epstopdf,
}
\usepackage{threeparttable}
\usepackage{booktabs}
\usepackage{bbold}

\usepackage[small
]{titlesec}

\usepackage[pdfpagemode={UseOutlines},bookmarks=true,bookmarksopen=true, bookmarksopenlevel=0,bookmarksnumbered=true,hypertexnames=false, colorlinks,linkcolor={blue},citecolor={blue},urlcolor={blue}, pdfstartview={FitV},unicode,breaklinks=true,backref=page]{hyperref}

\theoremstyle{plain}
\newtheorem{thm}{Theorem}
\newtheorem{prop}[thm]{Proposition}

\newtheorem{lem}[thm]{Lemma}
\newtheorem{cor}[thm]{Corollary}

\newtheorem{conj}[thm]{Conjecture}

\theoremstyle{definition}
\newtheorem{define}{Definition}

\theoremstyle{remark}

\newcommand\Ob{\operatorname{Ob}}

\newcommand\id{\operatorname{id}}
\newcommand\Hom{\operatorname{Hom}}
\newcommand\End{\operatorname{End}}
\newcommand\Hilb{\operatorname{Hilb}}
\newcommand\ev{\operatorname{ev}}

\newcommand\coev{\operatorname{coev}}
\newcommand\tr{\operatorname{tr}}

\newcommand\cob{\operatorname{Cob}_3}

\newcommand\Mod{\operatorname{Mod}}

\begin{document}
\maketitle

\small

\begin{abstract}
This paper presents a way to estimate the Heegaard genus of a $3$-manifold using the Turaev--Viro state sum TQFT. The Turaev--Viro state sum TQFT is derived from the modular category associated to the quantum group $U_q(\mathfrak{sl}_2)$, which is unitary for some $q$ by Wenzl. Hence by Turaev and Virelizier the corresponding TQFT is unitary. We modify a proof by Garoufalidis to give a lower bound of the Heegaard genus using a unitary TQFT, and then  use the software \textit{Regina} to provide some known counterexamples to the rank versus genus conjecture. 
\end{abstract}

\normalsize


\section{Overview}
%
%
%
%
%

For a closed orientable $3$-manifold $M$, the \textit{Heegaard genus} $g(M)$ of $M$,
is the minimal genus over all Heegaard splittings of $M$. The \textit{rank} $r(M)$ of $M$ is the rank of the fundamental group $\pi_1(M)$. Hence
$
r(M)\leq g(M)$. 

\begin{conj}
r(M)= g(M)? 
\end{conj}

This conjecture was proposed by Waldhausen \cite{waldhausen}. In 1984, Boileau--Zieschang \cite{BZ} discovered the first counterexamples among Seifert manifolds. Later, for any $n$, graph manifolds of genus $4n$ whose fundamental
group is $3n$-generated are constructed by Schultens and Weidmann  in \cite{sch-weid}. In 2009, Namazi and Souto \cite{namazi} showed that rank equals genus if the
gluing map of a Heegaard splitting is a high power of a generic pseudo-Anosov map. In 2013, Tao Li showed in \cite{taoli} that there are counterexamples among hyperbolic manifolds, using topological argument.

In this thesis, we apply the Turaev--Viro TQFT to the rank versus genus conjecture.  We obtain a lower bound of the Heegaard genus of a $3$-manifold, and provide some known counterexamples to the rank versus genus conjecture.

\subsection{Background}

Inspired by Witten, Reshetikhin and Turaev \cite{RT} constructed quantum invariants using representations of a quantum group $U_q(\mathfrak{sl}_2(\mathbb{C}))$. This invariant is part of a topological quantum field theory (TQFT). This was generalized to simple Lie algebras of type A, B, C and D in \cite{TW}. More generally, the TQFT can be constructed using a general modular category, as is explained in \cite{Tumodular} or \cite{Tu}. There is another approach via Kauffman bracket, see \cite{BHMV1992}, \cite{BHMV1995}. The book \cite{lusztig} is an introduction to quantum groups. 

In \cite{TV}, Turaev and Viro constructed a TQFT using quantum 6j-symbols associated to $U_q(\mathfrak{sl}_2(\mathbb{C}))$. More generally, the construction can be performed using a unimodal category, see \cite{Tu}, or using a finite semisimple spherical category of non-zero dimension, see \cite{BW}. This is called the Turaev--Viro--Barrett--Westbury invariant in the book by Turaev and Virelizier \cite{TuVi}. In 2018, Qingtao Chen and Tian Yang \cite{CY} defined the invariant for a manifold with boundary and proposed volume conjectures for the Reshetikhin--Turaev and the Turaev--Viro invariants. 

The process of constructing a modular category from $U_q(\mathfrak{sl}_2(\mathbb{C}))$ using the notion of tilting modules is explained in \cite{BK}. 
In \cite{BW}, Barrett and Westbury used the paper by Andersen and Paradowski \cite{AP}, and showed  that the quantum $\mathfrak{sl}_2$ can be used to produce a state sum TQFT. (This general procedure is needed in the proof of unitarity. )

Wenzl \cite{wenzl} showed that, if $\mathfrak{g}$ is a simple Lie algebra and $q=e^{\frac{\pi i}{ld}}$ where $d=m\in\{1,2,3\}$ is the ratio of the square lengths of a long and a short root and $l=\varkappa$ is larger than the dual Coxeter number, then the category for $U_q \mathfrak{g}$ is unitary. (It is shown in \cite{kirillov} that the category is Hermitian. ) Note that the ratio is $1$ for Lie type $A$. The dual Coxeter number of $A_n$ is $n+1$, and $\mathfrak{sl}_2$ corresponds to $A_1$. There is a survey \cite{rowell} on the general case.

It is shown in \cite{TuVi}, Appendix G, that a unitary fusion category ${\mathcal{C}}$ gives rise to a unitary state sum TQFT.

\subsection{A Lower Bound of the Heegaard Genus}

In \cite{garoufalidis}, Theorem 2.2 gives a lower bound of Heegaard genus using a unitary TQFT.  \textbf{This theorem is not correct in general. }Note that Theorem 2.2 uses Lemma 2.1, and  the author claims that Lemma 2.1 follows from a paper by Witten \cite{witten}. In \cite{witten}, there is an implicit condition on the TQFT, namely the space associated to $S^2$ must be one-dimensional. This condition is not mentioned in \cite{garoufalidis}.

We modify this  proof in \cite{garoufalidis} to give the bound. 

Parallel to Lemma 2.1 in \cite{garoufalidis}, our unitary TQFT $Z$ must have the following three additional properties: 

(1) If $M, N$ are closed $3$-manifolds, then $Z(M\# N)Z(S^3)=Z(M)Z(N)$.  

(2) $Z(S^2\times S^1)=1$. 

(3) $0<Z(S^3)<1$. 

Assuming the space associated to $S^2$ is one-dimensional, property (1) follows from a cut-and-paste argument (see Lemma 10.2 in \cite{TuVi}). This is true at least for the state sum TQFT. Lemma 13.6 in \cite{TuVi} states that, over any commutative
ring $\mathbb{k}$, the $\mathbb{k}$-module $|S^2|_{\mathcal{C}}$ is isomorphic to $\mathbb{k}$.  (Here ${\mathcal{C}}$ is a  spherical fusion $\mathbb{k}$-category. ) 

As for property (2), section 13.1.3 in \cite{TuVi} shows that $|S^1\times S^2|_{{\mathcal{C}}}=1$. 

For property (3), $|S^{3}|_{{\mathcal{C}}}=(\operatorname{dim}({\mathcal{C}}))^{-1}$. If  ${\mathcal{C}} $ is spherical, then $\operatorname{dim}({\mathcal{C}})=\sum_{i \in I}(\operatorname{dim}(i))^{2} $, where $I$ is a (finite) representative set of simple objects of ${\mathcal{C}}$. Note that the quantum dimension is not necessarily an integer. 
Here, we consider the original Turaev--Viro TQFT \cite{TV}, at $q=e^{\pi i/r}$, $r\geq 3$. Then 
\[
|S^3|=-\frac{\left(e^{\frac{i \pi }{r}}-e^{-\frac{i \pi }{r}}\right)^2}{2 r}=\frac{2 \sin ^2\left(\frac{\pi }{r}\right)}{r}\in (0,1). 
\]

Combining the results, we obtain a lower bound of the Heegaard genus. It is worth noting that Corollary 11.7 in \cite{TVi} gives the same bound.

We use the software \textit{Regina} to search for counterexamples. The confirmed ones have  appeared in \cite{BZ}.

Here is the structure of the thesis. Chapter 2 includes necessary definitions and the construction of the category needed, described in \cite{BK}. In Chapter 3 we construct the state sum TQFT and sketch a proof of unitarity, described in \cite{TuVi} and \cite{TVi}. 
In Chapter 4 we present the estimation and the computer search.


\section{Modular Category Associated to $U_q(\mathfrak{sl}_2)$}

\subsection{Definitions}

Our definitions are taken from \cite{TuVi}, \cite{TVi}, \cite{rowell}, and \cite{BK}. 

A \textit{monoidal category} is a category with a tensor product $\otimes$ and an
identity object $\mathbb{1}$ satisfying axioms that guarantee that the tensor product
is associative (up to specified isomorphism) and that $\mathbb{1}\otimes X\cong X\otimes \mathbb{1}\cong X$ for any object $X$.

Let  $\mathcal{C}=(\mathcal{C}, \otimes, \mathbb{1})$  and  $\mathcal{D}=(\mathcal{D}, \otimes^{\prime}, \mathbb{1}^{\prime})$  be monoidal categories. A \textit{monoidal functor} from  $\mathcal{C}$  to $ \mathcal{D}$  is a functor $ F: \mathcal{C} \rightarrow \mathcal{D}$  endowed with a morphism $ F_{0}: \mathbb{1}^{\prime} \rightarrow F(\mathbb{1})$  in  $\mathcal{D}$  and with a natural transformation
\[
F_{2}=\left\{F_{2}(X, Y): F(X) \otimes^{\prime} F(Y) \rightarrow F(X \otimes Y)\right\}_{X, Y \in \Ob(\mathcal{C})}
\]
satisfying some compatibility conditions. The monoidal functor is \textit{strong} if $F_0$ and $F_{2}(X, Y)$ are isomorphisms
for all $X, Y$.

Let $\mathbb{k}$ be a commutative ring. A $\mathbb{k}$\textit{-category} $\mathcal{C}$ is a category such that all $\Hom$ sets are left $\mathbb{k}$-modules and the composition of morphisms is $\mathbb{k}$-bilinear. 
An object $D$ of $\mathcal{C}$ is a \textit{direct sum} of objects $X_1, \dots, X_n$ if there is a family of morphisms $p_i:D\to X_i, q_i: X_i\to D$, such that 
\[
\id_D=\sum q_i p_i,\  p_i q_j=\delta_{ij}\id_{X_i}, \forall i,j. 
\]
An object $X$ of $\mathcal{C}$ is \textit{simple} if $\End_{\mathcal{C}}(X)$ is isomorphic to $\mathbb{k}$ as a $\mathbb{k}$-module. 

A \textit{monoidal $\mathbb{k}$-category} is a $\mathbb{k}$-category which is monoidal and such that monoidal product of morphisms is $\mathbb{k}$-bilinear. It is \textit{semisimple} if any object is isomorphic to a
direct sum of simple ones.

A monoidal category has \textit{(left) duality} if for any object $X$ there is an object $X^*$ (the \textit{left dual} of $X$)  and $$b_X=\coev_X: \mathbb{1}\to X\otimes X^*, \ d_X=\ev_X: X^*\otimes X\to \mathbb{1}$$
such that
\[
(\id_X\otimes d_X)(b_X\otimes \id_X)=\id_X, \ 
(d_X\otimes \id_{X^*})(\id_{X^*}\otimes b_X)=\id_{X^*}. 
\]
A \textit{left rigid category} is a
monoidal category with  a left duality. 
A \textit{rigid category} is a monoidal category which is  left
rigid and right rigid. Here a right duality is similarly defined using 
$$\widetilde{\ev}_X:X\otimes X^*\to \mathbb{1}, \quad \widetilde{\coev}_X:\mathbb{1}\to X^*\otimes X. $$
For $f\in \Hom(X,Y)$, let $$f^{*}=\left(d_{Y} \otimes \id_{X^{*}}\right)\left(\id_{Y^{*}} \otimes f \otimes \id_{X^{*}}\right)\left(\id_{Y^{*}} \otimes b_{X}\right). $$

A \textit{braiding} in a monoidal category is a collection of natural  isomorphisms $$c_{X,Y}: X\otimes Y\cong Y\otimes X$$ such that
\[
c_{X, Y \otimes Z} =\left(\id_{Y} \otimes c_{X, Z}\right)\left(c_{X, Y} \otimes \id_{Z}\right) , \quad
c_{X \otimes Y, Z} =\left(c_{X, Z} \otimes \id_{Y}\right)\left(\id_{X} \otimes c_{Y, Z}\right). 
\]
A braiding   is \textit{symmetric} if for all $X, Y$, $c_{Y,X}c_{X,Y}=\id$. In this case the category is \textit{symmetric}. 
A \textit{braided functor} between braided categories  $(\mathcal{C}, c)$  and  $(\mathcal{C}^{\prime}, c^{\prime})$  is a monoidal functor  $F: \mathcal{C} \rightarrow \mathcal{C}^{\prime}$  such that for all  $X, Y \in \operatorname{Ob}(\mathcal{C}) $,
\[
F_{2}(Y, X) c_{F(X), F(Y)}^{\prime}=F\left(c_{X, Y}\right) F_{2}(X, Y). 
\]
A
\textit{symmetric functor} is a braided monoidal functor between symmetric categories.


A \textit{twist} in a braided monoidal category is a natural family of isomorphisms $\theta_X:X\cong X$, such that 
\[
\theta_{X \otimes Y}=c_{Y, X} c_{X, Y}\left(\theta_{X} \otimes \theta_{Y}\right). 
\]

A \textit{ribbon category} is a braided monoidal category with a twist and a left duality such that $\theta_{X^*}=(\theta_X)^*$. 

A \textit{pivotal category} $\mathcal{C}$  is a rigid category with distinguished duality
such that the induced left and right dual functors coincide as monoidal functors. The \textit{left trace} of an endomorphism $f$ of an object $X$ is 
\[
\tr_{l}(f)=\ev_{X}\left(\id_{X^{*}} \otimes f\right) \widetilde{\coev}_{X} \in \End_{\mathcal{C}}(\mathbb{1}) .
\]
The \textit{right trace} of  $f$  is 
\[
\tr_{r}(f)=\widetilde{\ev}_{X}\left(f \otimes \id_{X^{*}}\right) \coev_{X} \in \End_{\mathcal{C}}(\mathbb{1}) .
\]
The \textit{left dimension} of an object $X$ is  $\dim_l(X) = \tr_l(\id_X)$. The \textit{right dimension} is  $\dim_r(X) = \tr_r(\id_X )$. 

A \textit{spherical category} is a pivotal category whose left and right traces coincide. All ribbon categories are spherical.

A \textit{pre-fusion $\mathbb{k}$-category} is a monoidal $\mathbb{k}$-category $\mathcal{C}$ such that there is a set $I$ of
simple objects of $\mathcal{C}$ such that $\mathbb{1}\in I$, $\Hom_{\mathcal{C}}(i, j) = 0$ for any distinct $i, j \in I$, and every object of $\mathcal{C}$ is a direct sum of a finite family of elements of $I$. 

A \textit{fusion $\mathbb{k}$-category} is a rigid pre-fusion $\mathbb{k}$-category which has only a finite number of isomorphism classes of simple objects. 

The \textit{dimension} of a pivotal
fusion $\mathbb{k}$-category $\mathcal{C}$ is defined by $$\dim(\mathcal{C})=\sum_{i \in I} \dim_{l}(i) \dim_{r}(i) \in \mathbb{k}. $$

Let $\mathcal{C}$ be a ribbon fusion $\mathbb{k}$-category. For any  $i, j \in I $, set
$$S_{i, j}=\tr\left(c_{j, i} c_{i, j}\right) \in \End_{\mathcal{C}}(\mathbb{1})=\mathbb{k}. $$  The category $\mathcal{C}$ is said to be \textit{modular} if this  $S$-matrix is invertible over $\mathbb{k}$. The modularity of $\mathcal{C}$ implies that $\dim(\mathcal{C})$ is invertible in $\mathbb{k}$. 

A \textit{Hermitian fusion category} is a spherical fusion category $\mathcal{C}$ over $\mathbb{C}$ 
endowed with antilinear homomorphisms
\[
\left\{f \in \Hom_{\mathcal{C}}(X, Y) \mapsto  \bar{f} \in \Hom_{\mathcal{C}}(Y, X)\right\}_{X, Y \in \operatorname{Ob}(\mathcal{C})}
\]
such that $\bar{\bar{f}}=f, \overline{gf}=\bar{f}\bar{g}, \overline{f\otimes g}=\bar{f}\otimes \bar{g}, \overline{\ev_X}=\widetilde{\coev}_X, \overline{\coev_X}=\widetilde{\ev}_X$. 

A \textit{unitary fusion category} is a Hermitian fusion category $\mathcal{C}$ such that 
$\tr(f\bar{f}) > 0$ for any non-zero morphism $f$ in $\mathcal{C}$. The dimension of
a unitary fusion category is a positive real number.

\subsection{Representations of $U_q(\mathfrak{sl}_2)$}
%
%
%
%

We will focus on the case where $\mathfrak{g}=\mathfrak{sl}_2$, $q=e^{\pi i/r}$, $r\geq 3$. We follow \cite{BK}.

\begin{define}
A \textit{Hopf algebra} $A$ over a field $\mathbb{k}$ is  a bialgebra   with 
an algebra anti-isomorphism $\gamma:A\to A$, called the antipode, satisfying 
$ \mu (\id \otimes \gamma)\Delta=\epsilon=\mu(\gamma \otimes \id )\Delta$. 
\end{define}

The category of finite-dimensional representations of $A$ is a rigid monoidal 
category.

Let $\mathfrak{g}$  be a finite-dimensional simple Lie algebra over  $\mathbb{C} $,
 and $\mathfrak{h}$  be its Cartan subalgebra. 
Let $\Delta \subset \mathfrak{h}^{*}$ be   the root system, 
 $\Pi=\left\{\alpha_{1}, \dots, \alpha_{r}\right\} \subset \Delta$  be the set of simple roots,
 $h_{i}=\alpha_{i}^{\vee} \in \mathfrak{h}$  be the coroots, and 
 $A=\left(a_{i j}\right)$  be the Cartan matrix,  $a_{i j}=\left(\alpha_{i}^{\vee}, \alpha_{j}\right) $. 

Let  $P \subset \mathfrak{h}^{*}$  be the weight lattice, 
 $Q \subset \mathfrak{h}^{*}$  be the root lattice, and 
 $Q^{\vee} \subset \mathfrak{h}$ be   the dual root lattice (coroot lattice). 
Let  $\langle\langle\cdot, \cdot\rangle\rangle$  be an invariant bilinear form on  $\mathfrak{g}$  normalized by  $\langle\langle \alpha, \alpha\rangle\rangle=2$  for short roots  $\alpha $. Then  $d_{i}:=\langle\langle\alpha_{i}, \alpha_{i}\rangle\rangle / 2 \in \mathbb{Z}_{+} $ for all  $i=1, \dots, r $. 
Let  $\mathbb{C}_{q}$   be the field  $\mathbb{C}\left(q^{1 /|P / Q|}\right) $ where  $q$  is a formal variable.

\begin{define}
The quantum group  $U_{q}(\mathfrak{g})$  is the associative algebra over  $\mathbb{C}_{q} $ with generators  $e_{i}, f_{i}(i=1, \dots, r), q^{h}\left(h \in Q^{\vee}\right)$  and relations
\[
 q^{h^{\prime}} q^{h^{\prime \prime}} =q^{h^{\prime}+h^{\prime \prime}}, q^{0}=1, \quad h^{\prime}, h^{\prime \prime} \in Q^{\vee} , 
\]
\[
 q^{h} e_{i} q^{-h} =q^{\left(h, \alpha_{i}\right)} e_{i}, q^{h} f_{i} q^{-h}=q^{-\left(h, \alpha_{i}\right)} f_{i} , 
{\left[e_{i}, f_{j}\right] } =\delta_{i j} \frac{q^{d_{i} h_{i}}-q^{-d_{i} h_{i}}}{q^{d_{i}}-q^{-d_{i}}} , 
\]
\[
\sum_{k = 0}^{1-a_{i j}}(-1)^{k}\left[\begin{array}{c}
1-a_{i j} \\
k
\end{array}\right]_{i} e_{i}^{1-a_{i j}-k} e_{j} e_{i}^{k} = 0, \quad i \neq j, 
\]
\[
\sum_{k = 0}^{1-a_{i j}}(-1)^{k}\left[\begin{array}{c}
1-a_{i j} \\
k
\end{array}\right]_{i} f_{i}^{1-a_{i j}-k} f_{j} f_{i}^{k} = 0, \quad i \neq j, 
\]
where
\[
{\left[\begin{array}{l}
n \\
k
\end{array}\right]_{i}: = \frac{[n]_{i} !}{[k]_{i} ![n-k]_{i} !}, \quad[n]_{i} !: = [1]_{i}[2]_{i} \cdots[n]_{i}, \quad[n]_{i}: = \frac{q^{d_{i} n}-q^{-d_{i} n}}{q^{d_{i}}-q^{-d_{i}}}}. 
\]
\end{define}

This is a Hopf algebra with
\[
\begin{array}{c}
\Delta\left(q^{h}\right)=q^{h} \otimes q^{h} , \\
\Delta\left(e_{i}\right)=e_{i} \otimes q^{d_{i} h_{i}}+1 \otimes e_{i} , \\
\Delta\left(f_{i}\right)=f_{i} \otimes 1+q^{-d_{i} h_{i}} \otimes f_{i} , \\
\epsilon\left(q^{h}\right)=1, \quad \epsilon\left(e_{i}\right)=\epsilon\left(f_{i}\right)=0 , \\
\gamma\left(q^{h}\right)=q^{-h}, \quad \gamma\left(e_{i}\right)=-e_{i} q^{-d_{i} h_{i}}, \quad \gamma\left(f_{i}\right)=-q^{d_{i} h_{i}} f_{i}. 
\end{array}
\]

Let  $\mathcal{A}=\mathbb{Z}\left[q^{\pm 1 /|P / Q|}\right]$  and let  $U_{q}(\mathfrak{g})_{\mathbb{Z}}$  be the  $\mathcal{A} $-subalgebra of  $U_{q}(\mathfrak{g}) $ generated by the elements
\[
 e_{i}^{(n)}=\frac{e_{i}^{n}}{[n]_{i} !}, \quad f_{i}^{(n)}=\frac{f_{i}^{n}}{[n]_{i} !}, \quad q^{h}, \quad n=1,2, \dots, i=1, \dots, r, h \in Q^{\vee}. 
\]

Fix  $\varkappa \in \mathbb{C}^{\times}$ and consider  $\mathbb{C}$  as an  $\mathcal{A} $-module via the homomorphism
\[
\mathcal{A} \rightarrow \mathbb{C}, \quad q^{a} \mapsto e^{\frac{a \pi i}{m \varkappa}}, \quad  m:=\operatorname{max}d_i. 
\]
Set
\[
U_{q}(\mathfrak{g})|_{q=e^{\frac{\pi i}{m \varkappa}}}:=U_{q}(\mathfrak{g})_{\mathbb{Z}} \otimes_{\mathcal{A}} \mathbb{C}. 
\]

We  abbreviate 
$
U_{q}(\mathfrak{g})|_{q=e^{\frac{\pi i}{m \varkappa}}}
$
to 
$
U_{q}(\mathfrak{g})
$.

The  category $\mathcal{C}(\mathfrak{g},\varkappa)$ of finite-dimensional representations of  $U_{q}(\mathfrak{g})$ over $\mathbb{C}$ possessing a weight decomposition
\[
V=\bigoplus_{\lambda \in P} V^{\lambda},\quad q^{h}|_{V^{\lambda}}=q^{(h, \lambda)} \mathrm{id}_{V^{\lambda}},  
\]
\[
e_{i}^{(n)}\left(V^{\lambda}\right) \subset V^{\lambda+n \alpha_{i}}, \quad f_{i}^{(n)}\left(V^{\lambda}\right) \subset V^{\lambda-n \alpha_{i}}, 
\]
is a ribbon category.

The \textit{Weyl modules} of $U_q(\mathfrak{sl}_2)$ are
$
V_{n}=\oplus_{i=0}^{n} \mathbb{C} v_{i},  n \in \mathbb{Z}_{+} 
$, 
where  $v_{0}$  is the highest weight vector and  $v_{i}=f^{(i)} v_{0} $. The action  is
\[
q^{h} v_{i}=q^{n-2 i} v_{i}, \quad e v_{i}=[n-i+1] v_{i-1}, \quad f v_{i}=[i+1] v_{i+1}. 
\]

The module  $V_{n}$  is irreducible for  $n<\varkappa $, and $ \operatorname{dim}_{q} V_{n}=[n+1]=0 $  if and only if  $\varkappa$  divides  $n+1 $.  Hence  for  $0 \leq n \leq \varkappa-2$, $V_{n}$  is irreducible and  $\operatorname{dim}_{q} V_{n} \neq 0 $. 
Correspondingly, let $C:=\{0,\frac{1}{2}, \dots, \frac{\varkappa-2}{2}\}$ be the  set of weights, where $\mathfrak{g}=\mathfrak{sl}_2$, $q=e^{\pi i/\varkappa}$, $\varkappa\geq 3$. From now on we index the Weyl modules by their weights in $\frac{1}{2}\mathbb{Z}$.

A module  $T$  over  $U_q(\mathfrak{sl}_2)$  is called \textit{tilting} if both  $T$  and  $T^{*}$  have composition series with factors Weyl modules. Let  $\mathcal{T}$  be the full subcategory of  $\mathcal{C}(\mathfrak{sl}_2, \varkappa)$  consisting of all tilting modules. 
If  $\lambda \in C$  then  the module  $V_{\lambda}$  is tilting. 
It is shown in \cite{AP} that the category of tilting modules  $\mathcal{T}$  is closed under  $*, \oplus, \otimes$  and direct summands. It is a ribbon category.

A tilting module  $T$  is called \textit{negligible} if  $\operatorname{tr}_{q} f=0 $ for any $ f \in \operatorname{End} T $.  A morphism  $f: T_{1} \rightarrow T_{2}$  is called \textit{negligible} if  $\operatorname{tr}_{q}(f g)=0$  for all  $g: T_{2} \rightarrow T_{1} $.

Let  ${\mathcal{C}}^{\text {int }} := {\mathcal{C}}^{\text {int }}(\mathfrak{sl}_2, \varkappa)$  be the category with objects tilting modules and morphisms
\[
\operatorname{Hom}_{{\mathcal{C}}^{\text {int }}}(V, W)=\operatorname{Hom}_{\mathcal{T}}(V, W) / \text{negligible morphisms. }
\]
Then  ${\mathcal{C}}^{\text {int }}$  is a ribbon category. 
Any object  $V$  in  ${\mathcal{C}}^{\text {int }}$  is isomorphic to  $\bigoplus_{\lambda \in C} n_{\lambda} V_{\lambda} $. 
${\mathcal{C}}^{\text {int }}$  is a semisimple abelian category. 
Indeed,   ${\mathcal{C}}^{\text {int }}$  is a modular tensor category with simple objects  $V_{\lambda}   (\lambda \in C) $.

%
%

\subsection{Unitarity of the category}

For a survey on the general case, see \cite{rowell}.  
We will focus on the case where $\mathfrak{g}=\mathfrak{sl}_2$, $q=e^{\pi i/r}$, $r\geq 3$. 

The invertibility of $S$-matrix is shown in \cite{TW}. Their work is based on the notion of quasimodular Hopf algebra, which is a weakened version of modular Hopf algebra introduced in the paper \cite{RT}.

In \cite{kirillov}, Hermitian structure on the category is defined in Section 4. Note that the notation in \cite{kirillov} is different. 
In \cite{wenzl}, positivity is proved by Wenzl. Xu \cite{xu} independently showed some of the cases covered by Wenzl.

\begin{thm}[\cite{wenzl}, \cite{xu}]
The categories  ${\mathcal{C}}^{\text {int }}(\mathfrak{sl}_2, \varkappa)$  are unitary when    $q=e^{\pi i / \varkappa} $. 
\end{thm}


\section{Turaev--Viro TQFT}

%
%
%

\subsection{TQFT}

Manifolds considered are assumed to be oriented. 

An \textit{$n$-dimensional cobordism} is a quadruple $(M, \Sigma_0, \Sigma_1, h)$, where $M$ is a compact oriented $n$-manifold, $\Sigma_0, \Sigma_1$ are closed oriented $(n-1)$-manifolds, and $h:(-\Sigma_0)\sqcup \Sigma_1 \to \partial M$ is an orientation preserving homeomorphism. 
Two cobordisms  $\left(M, \Sigma_{0}, \Sigma_{1}, h\right)$  and  $\left(M^{\prime}, \Sigma_{0}, \Sigma_{1}, h^{\prime}\right)$  between  $\Sigma_{0}$  and  $\Sigma_{1}$  are \textit{homeomorphic} if there is an orientation preserving homeomorphism  $\alpha: M \rightarrow M^{\prime}$  such that  $h^{\prime}=\alpha h $. 

Let $\cob $ be the category with objects closed oriented $2$-manifolds and morphisms cobordisms between them considered up to homeomorphisms. It is a symmetric monoidal category,  where the
monoidal product  is the ordered disjoint union of manifolds and cobordisms.   Each closed oriented $3$-manifold represents a morphism $\emptyset\to \emptyset$.

Let $\Mod_{\mathbb{k}}$ be  the symmetric monoidal category of 
$\mathbb{k}$-modules. 

\begin{define}
A $3$-dimensional \textit{Topological Quantum Field Theory} (TQFT) is a
symmetric strong monoidal functor
\[
Z:\cob \to \Mod_{\mathbb{k}}. 
\]
\end{define}

A \textit{conjugation} in a monoidal category $\mathcal{C}$ is a family of maps
\[
 \left\{f \in \operatorname{Hom}_{\mathcal{C}}(X, Y) \rightarrow \bar{f} \in \operatorname{Hom}_{\mathcal{C}}(Y, X)\right\}_{X, Y \in \operatorname{Ob}(\mathcal{C})} 
\]
compatible with  the associativity and unitality constraints of  $\mathcal{C} $, 
such that for all morphisms  $f, g$  in  $\mathcal{C}$  and all  $X \in \operatorname{Ob}(\mathcal{C}) $
\[
\begin{array}{cccc}
\overline{g f}=\bar{f} \bar{g}, & \overline{\id_{X}}=\id_{X}, & \bar{\bar{f}}=f, & \overline{f \otimes g}=\bar{f} \otimes \bar{g}.  \\
\end{array}
\]

The category $\cob$ has a conjugation
\[
\left(M, \Sigma_{0}, \Sigma_{1}, h:\left(-\Sigma_{0}\right) \sqcup \Sigma_{1} \rightarrow \partial M\right)
\]
\[
\mapsto
\left(-M, \Sigma_{1}, \Sigma_{0}, h P:\left(-\Sigma_{1}\right) \sqcup \Sigma_{0} \rightarrow-\partial M=\partial(-M)\right),
\]
where  $P$  is the orientation reversing permutation homeomorphism
\[
\left(-\Sigma_{1}\right) \sqcup \Sigma_{0} \simeq\left(-\Sigma_{0}\right) \sqcup \Sigma_{1} .
\]

Let $\Hilb$ be the category of finite-dimensional Hilbert spaces over $\mathbb{C}$ and $\mathbb{C}$-linear
homomorphisms between them. This is a symmetric ribbon $\mathbb{C}$-category. The conjugation is given by assigning to every morphism its hermitian
adjoint.

\begin{define}
A \textit{unitary $3$-dimensional TQFT} is a symmetric strong
monoidal functor 
\[
Z : \cob \to \Hilb
\]
commuting with the conjugations in $\cob$ and 
in $\Hilb$.  
\end{define}

Note that there is a forgetful functor $\Hilb\to \Mod_{\mathbb{C}}$. 
Composing such a functor $Z$ with the 
forgetful functor, we obtain a $3$-dimensional TQFT. We call $Z$ a \textit{unitary lift} of the latter TQFT.

The category  $\cob$  has a subcategory isomorphic to  the category of closed oriented surfaces and isotopy classes of orientation-preserving homeomorphisms. Any such homeomorphism  $f: \Sigma \rightarrow \Sigma^{\prime}$  determines a morphism   $$\left(C=\Sigma^{\prime} \times[0,1], \Sigma , \Sigma^{\prime}, h:(-\Sigma) \sqcup \Sigma^{\prime} \simeq \partial C\right) $$ in $\cob$, where  $h(x)=(f(x), 0)$  for  $x \in \Sigma$  and  $h\left(x^{\prime}\right)=\left(x^{\prime}, 1\right)$  for  $x^{\prime} \in \Sigma^{\prime} $. Restricting a TQFT  $Z: \cob \rightarrow  \Mod_{\mathbb{k}}$  to this subcategory, we obtain the action of homeomorphisms, and in particular the action of mapping class groups. We denote the induced map by $f_\#:Z(\Sigma) \rightarrow Z(\Sigma^{\prime})$. It follows that $(f^{-1})_\#=(f_\#)^{-1}$. 

If, in addition, $Z$ is unitary, then 
\begin{prop}
$\langle f_\#x,f_\# y \rangle=\langle x,y \rangle$. 
\end{prop}

\proof{The adjoint of $f_\#$ is the same as the inverse of $f_\#$. Indeed there is a homeomorphism from the conjugation of $\left(C=\Sigma^{\prime} \times[0,1], \Sigma , \Sigma^{\prime}, h:(-\Sigma) \sqcup \Sigma^{\prime} \simeq \partial C\right) $ to the cobordism associated to $f^{-1}$. 
\qedhere}

Theorem 10.3.1 in \cite{Tu} is a similar result.

\subsection{Construction of Turaev--Viro state sum  TQFT}

%
%
%

We sketch a construction following \cite{TuVi} and \cite{TVi}. 

Fix a  spherical fusion $\mathbb{k}$-category $\mathcal{C}$ of invertible dimension. Fix a (finite) representative set $I$ of simple objects of $\mathcal{C}$.

\subsubsection{Colored graphs in surfaces}

A \textit{graph} $G$ in an oriented surface $\Sigma$ is  a finite graph without
isolated vertices embedded in $\Sigma$. A graph $G$ is \textit{$\mathcal{C}$-colored}, if each edge of $G$ is oriented and endowed with an object
of $\mathcal{C}$ called the \textit{color} of the edge.

For a vertex $v$ of a $\mathcal{C}$-colored graph $G$, the set $E_v$ of half-edges of $G$ incident to $v$ with cyclic order induced by the opposite orientation of $\Sigma$, and the color $c=c_v: E_v\to \Ob(\mathcal{C})$, the orientation $\epsilon=\epsilon_v:E_v\to \{+, -\}$, together form a \textit{cyclic $\mathcal{C}$-set}. The sign is $+$ if $e\in E_v$ is oriented
towards $v$. 

Starting from $e=e_1\in E_v$, let
\[
H_{e}=\operatorname{Hom}_{\mathcal{C}}\left(\mathbb{1}, c\left(e_{1}\right)^{\epsilon\left(e_{1}\right)} \otimes \cdots \otimes c\left(e_{n}\right)^{\epsilon\left(e_{n}\right)}\right), 
\]
where $X^+=X$, $X^-=X^*$. 


%

This module can be made into the \textit{symmetrized multiplicity module}  $H(E_v)$,   independent of the starting point chosen.  

Set  $H_{v}(G)=H\left(E_{v}\right)$  and  
\[
H(G)=\otimes_{v} H_{v}(G). 
\]

Orient the plane  $\mathbb{R}^{2}$  counterclockwise. Let  $G$  be a  $\mathcal{C} $-colored graph in  $\mathbb{R}^{2} $. 
Pick any  $\alpha_{v} \in H_{v}(G)$  and replace  $v$  by a box colored with the corresponding element in the $\Hom $ set. This transforms  $G$  into a planar diagram which determines, by the graphical calculus, an element $\mathbb{F}_{\mathcal{C}}(G)\left(\otimes_{v} \alpha_{v}\right)$  of  $\operatorname{End}_{\mathcal{C}}(\mathbb{1})=\mathbb{k}$. See \cite{Tu}, \cite{TuVi} or \cite{TVi} for an introduction to the graphical calculus. 
This procedure defines a vector  
\[
\mathbb{F}_{\mathcal{C}}(G) \in   H(G)^{\star}=\operatorname{Hom}_{\mathbb{k}}(H(G), \mathbb{k}). 
\]

As  $\mathcal{C}$ is spherical,  the invariant $\mathbb{F}_{\mathcal{C}}$  generalizes to graphs in the $2$-sphere.

\subsubsection{State sums on skeletons of $3$-manifolds}

A \textit{$2$-polyhedron} is  a compact topological space that can be triangulated using a finite number of simplices of dimension  $\leq 2$  so that all $0 $-simplices and $1$-simplices are faces of $2 $-simplices. For a $2$-polyhedron  $P$,   let   $\operatorname{Int}(P)$ be  the subspace  consisting of all points having a neighborhood homeomorphic to  $\mathbb{R}^{2} $.  

A \textit{stratification} of a $2$-polyhedron  $P$  is a graph $ G$  embedded in  $P$  so that  $P \backslash \operatorname{Int}(P) \subset G $. The vertices and edges of $P^{(1)}:=G$ are called respectively the vertices and edges of $P$.

Cutting a stratified $2$-polyhedron  $P$  along the graph  $P^{(1)} \subset P $, we obtain a compact surface  $\tilde{P}$  with interior  $P \backslash P^{(1)} $.  The (connected) components of  $\tilde{P}$  are called the \textit{regions} of  $P $.  We let  $\operatorname{Reg}(P)$  be the (finite) set of all regions of  $P $. 

A \textit{branch} of a stratified $2$-polyhedron  $P$  at a vertex  $x$    is a germ at  $x$  of an adjacent region.  A \textit{branch} of  $P$  at an edge  $e$    is a germ at  $e$  of an adjacent region.  The set of branches of  $P$  at  $e$  is denoted  $P_{e} $. The number of elements of  $P_{e}$  is called the \textit{valence} of  $e $.

The edges of  $P$  of valence $1$ and their vertices form a graph called the \textit{boundary} of  $P$  and denoted  $\partial P $. We say that  $P$  is \textit{orientable (resp. oriented)} if all regions of  $P$  are orientable (resp. oriented).


A \textit{skeleton} of a closed $3$-manifold  $M$  is an oriented stratified $2$-polyhedron  $P \subset M$  such that  $\partial P=\emptyset$  and  $M \backslash P$  is a disjoint union of open $3$-balls. 

Any vertex  $x$  of a skeleton  $P \subset M$  has a closed ball neighborhood  $B_{x} \subset M$  such that  $\Gamma_{x}=P \cap \partial B_{x}$  is a finite non-empty graph and  $P \cap B_{x}$  is the cone over  $\Gamma_{x} $.  The pair  $\left(\partial B_{x}, \Gamma_{x}\right)$   is the \textit{link} of  $x$  in  $(M, P) $. If  $M$  is oriented, then we endow  $\partial B_{x}$  with orientation induced by that of  $M$  restricted to  $M \backslash \operatorname{Int}\left(B_{x}\right) $.

Let $M$ be a closed oriented $3$-manifold. Fix a skeleton $P$ of $M$. 
For a map  $c: \operatorname{Reg}(P) \rightarrow I $ and an oriented edge  $e$  of  $P$,  we define a  $\mathbb{k} $-module  $$H_{c}(e)=H\left(P_{e}\right),  $$ where  $P_{e}$  is the cyclic  $\mathcal{C} $-set of branches of  $P$  at  $e$.  If  $e^{\mathrm{op}}$  is the same edge with opposite orientation, then  there is  a contraction  $$*_{e}: H_{c}(e)^{\star} \otimes H_{c}\left(e^{\mathrm{op}}\right)^{\star} \rightarrow \mathbb{k}. $$

The link of a vertex  $x \in P$  determines a  $\mathcal{C} $-colored graph  $\Gamma_{x}$  in  $\partial B_{x} \simeq S^{2} $. 
Hence there is  a tensor  $\mathbb{F}_{\mathcal{C}}\left(\Gamma_{x}\right) \in H_{c}\left(\Gamma_{x}\right)^{\star}$. Note that  $ H_{c}\left(\Gamma_{x}\right)=\otimes_{e} H_{c}(e) $, where  $e$  runs over all edges of  $P$  incident to  $x$  and oriented away from  $x$. 
Hence   $\otimes_{x} \mathbb{F}_{\mathcal{C}}\left(\Gamma_{x}\right)  \in \otimes_{e} H_{c}(e)^{\star} $, where  $e$  runs over all oriented edges of  $P $. Set  $*_{P}=\otimes_{e} *_{e}: \otimes_{e} H_{c}(e)^{\star} \rightarrow \mathbb{k} $.

Define 
\[
|M|_{\mathcal{C}}=(\operatorname{dim}(\mathcal{C}))^{-|P|} \sum_{c}\left(\prod_{r \in \operatorname{Reg}(P)}(\operatorname{dim} c(r))^{\chi(r)}\right) *_{P}\left(\otimes_{x} \mathbb{F}_{\mathcal{C}}\left(\Gamma_{x}\right)\right) \in \mathbb{k}, 
\]
where  $|P|$  is the number of components of $ M \backslash P $, $c$ runs over all maps  $\operatorname{Reg}(P) \rightarrow I $, and  $\chi(r)$  is the Euler characteristic of  $r$.

\begin{thm}[\cite{TVi}, \cite{TuVi}]
The invariant $|M|_{\mathcal{C}}$ is independent of $P$ chosen. 
\end{thm}


\subsubsection{The state sum TQFT}

Let  $M$  be a compact $3 $-manifold with boundary. Let  $G$  be an oriented graph in  $\partial M$  such that all vertices of  $G$  have valence  $\geq 2 $.  A \textit{skeleton}  of the pair  $(M, G)$  is an oriented stratified $2$-polyhedron  $P \subset M$  satisfying certain requirements including $P \cap \partial M=\partial P=G $. 

%
%
%
%
%
%

%

For any compact oriented $3$-manifold  $M $ and any  $I $-colored graph  $G$  in  $\partial M $, we define a topological invariant  $|M, G| \in \mathbb{k}$  as follows. Pick a skeleton  $P \subset M $ of the pair  $(M, G) $. Pick a  map $c: \operatorname{Reg}(P) \rightarrow I$  extending the coloring of  $G$. 

Let  $E_{0}$  be the set of oriented edges of $ P$  with both endpoints in  $\operatorname{Int}(M) $, and let $ E_{\partial}$  be the set of edges of $ P$  with exactly one endpoint in $ \partial M $  oriented towards this endpoint. We have
\[
\otimes_{e \in E_{\partial}} H_{c}(e)^{\star}=\otimes_{v} H_{v}\left(G^{\mathrm{op}} ;-\partial M\right)^{\star}=H\left(G^{\mathrm{op}} ;-\partial M\right)^{\star} , 
\]
where $G^{\text {op }}$  in  $-\Sigma$ is   obtained by reversing orientation in all edges of  $G$  and in  $\Sigma$. 

There is a contraction homomorphism
\[
*_{P}: \otimes_{e \in E_{0} \cup E_{\partial}} H_{c}(e)^{\star} \rightarrow \otimes_{e \in E_{\partial}} H_{c}(e)^{\star}=H\left(G^{\mathrm{op}} ;-\partial M\right)^{\star} .
\]

The tensor product  $\otimes_{x} \mathbb{F}_{\mathcal{C}}\left(\Gamma_{x}\right)$  over all vertices  $x$  of  $P$  lying in  $\operatorname{Int}(M)$  is a vector in  $\otimes_{e \in E_{0} \cup E_{\partial}} H_{c}(e)^{\star} $.

Set
\[
|M, G|=(\operatorname{dim}(\mathcal{C}))^{-|P|} \sum_{c}\left(\prod_{r \in \operatorname{Reg}(P)}(\operatorname{dim} c(r))^{\chi(r)}\right) *_{P}\left(\otimes_{x} \mathbb{F}_{\mathcal{C}}\left(\Gamma_{x}\right)\right), 
\]
where  $|P|$  is the number of components of  $M \backslash P $, $c$ runs over all maps  $\operatorname{Reg}(P) \rightarrow I$  extending the coloring of  $G $. 

\begin{thm}[\cite{TVi}, \cite{TuVi}]
$ |M, G| \in   H\left(G^{\mathrm{op}} ;-\partial M\right)^{\star}$  does not depend on the choice of  $P $. 
\end{thm}


Let  $\Sigma_{0}, \Sigma_{1}$  be closed oriented surfaces, and let  $f: \Sigma_{0} \rightarrow \Sigma_{1}$  be a morphism in  $\mathrm{Cob}_{3}$  represented by a pair  $(M, h)$. For  $I $-colored graphs  $G_{0} \subset \Sigma_{0}$  and  $G_{1} \subset \Sigma_{1}$,  
\[
\left|M, h\left(G_{0}^{\mathrm{op}} \cup G_{1}\right)\right|_{\mathcal{C}} \in H\left(h\left(G_{0}^{\mathrm{op}} \cup G_{1}\right) ; \partial M\right)  \simeq \operatorname{Hom}_{\mathbb{k}}\left(H\left(G_{0}\right), H\left(G_{1}\right)\right). 
\]

Set
\[
\left|f, G_{0}, G_{1}\right|=\frac{(\operatorname{dim}(\mathcal{C}))^{\left|\Sigma_{1} \backslash G_{1}\right|}}{\operatorname{dim}\left(G_{1}\right)}\left|M, h\left(G_{0}^{\mathrm{op}} \cup G_{1}\right)\right|_{\mathcal{C}}: H\left(G_{0}\right) \rightarrow H\left(G_{1}\right),
\]
where  $\left|\Sigma_{1} \backslash G_{1}\right| $ is the number of components of  $\Sigma_{1} \backslash G_{1}$  and  $\operatorname{dim}\left(G_{1}\right)$  is the product over all edges of  $G_{1}$  of the dimensions of their colors. 




A \textit{skeleton} of a closed surface $ \Sigma $ is  an oriented graph  $G$  embedded in  $\Sigma$  such that all vertices of  $G$  have valence  $\geq 2$  and all components of  $\Sigma \backslash G$  are open disks.  We let  $\operatorname{col}(G) $ be the set of all maps from the set of edges of $ G$  to $ I $.

For a closed oriented surface  $\Sigma$  and a skeleton  $G \subset \Sigma $, consider the  $\mathbb{k} $-module
\[
|\Sigma, G|^{\circ}=\bigoplus_{c \in \operatorname{col}(G)} H((G, c) ; \Sigma) .
\]
For a morphism  $f: \Sigma_{0} \rightarrow \Sigma_{1}$  in  $\mathrm{Cob}_{3}$  and skeletons  $G_{0} \subset \Sigma_{0}$  and  $G_{1} \subset \Sigma_{1} $, consider the  $\mathbb{k} $-linear homomorphism
\[
\left|f, G_{0}, G_{1}\right|^{\circ}:\left|\Sigma_{0}, G_{0}\right|^{\circ} \rightarrow\left|\Sigma_{1}, G_{1}\right|^{\circ}
\]
whose restriction to every summand  $H\left(\left(G_{0}, c_{0}\right) ; \Sigma_{0}\right) $ of $ \left|\Sigma_{0}, G_{0}\right|^{\circ} $ is equal to
$
\sum_{c_{1} \in \operatorname{col}\left(G_{1}\right)}\left|f,\left(G_{0}, c_{0}\right),\left(G_{1}, c_{1}\right)\right|
$. 
%
%
%


For any skeletons  $G, G^{\prime}$  of a closed oriented surface  $\Sigma $, set
\[
p\left(G, G^{\prime}\right)=\left|\operatorname{id}_{\Sigma}, G, G^{\prime}\right|^{\circ}:|\Sigma, G|^{\circ} \rightarrow\left|\Sigma, G^{\prime}\right|^{\circ}, 
\]
\[
|\Sigma, G|=\operatorname{Im}(p(G, G)). 
\]

The  $\mathbb{k} $-linear homomorphism  $p\left(G, G^{\prime}\right)$  restricts to an isomorphism  $|\Sigma, G| \rightarrow\left|\Sigma, G^{\prime}\right|$  and the family  $\left(\{|\Sigma, G|\}_{G},\left\{p\left(G, G^{\prime}\right)\right\}_{G, G^{\prime}}\right)$  is a projective system. The projective limit
\[
|\Sigma|=\lim _{\longleftarrow}|\Sigma, G|
\]
is a  $\mathbb{k} $-module depending only on  $\Sigma $.

We  associate with each morphism  $f: \Sigma_{0} \rightarrow \Sigma_{1}$  in $\cob$  a homomorphism  $|f|:\left|\Sigma_{0}\right| \rightarrow\left|\Sigma_{1}\right| $. 
For each skeleton  $G$  of  $\Sigma $, we have a  $\mathbb{k} $-linear cone isomorphism  $\tau_{G}:|\Sigma| \rightarrow|\Sigma, G|$. 
Pick any skeletons  $G_{0} \subset \Sigma_{0}$  and  $G_{1} \subset \Sigma_{1} $. 
Set
\[
|f|=\tau_{G_{1}}^{-1} \circ\left|f, G_{0}, G_{1}\right|^{\circ} \circ \tau_{G_{0}}:\left|\Sigma_{0}\right| \rightarrow\left|\Sigma_{1}\right| .
\]
This homomorphism does not depend on the choice of  $G_{0}$  and  $G_{1} $.


\begin{thm}[\cite{TVi}, \cite{TuVi}]
The functor $|\cdot|=|\cdot|_{\mathcal{C}} $  is a $3$-dimensional TQFT. 
\end{thm}

\subsection{Unitarity}

Let $\mathbb{k}$ be an algebraically closed field. 
 Let  $\mathcal{C}$  be a unitary fusion category (over $\mathbb{C}$). Then   $\operatorname{dim}(\mathcal{C}) \neq 0 $. 

\begin{thm}[Theorem G.1 in \cite{TuVi}, Corollary 11.6 in \cite{TVi}]
The state sum TQFT  $|\cdot|_{\mathcal{C}}$  has a unitary lift.
\end{thm}

The proof is involved. Recall the construction of the  Reshetikhin--Turaev extended TQFT $\tau_{\mathcal{B}}$ associated to a modular category $\mathcal{B}$ over $\mathbb{k}$ equipped with a distinguished square root
of $\dim(\mathcal{B})$, described in detail in \cite{Tu}. In this case unitarity is proved in \cite{Tu}, Chapter IV, section 11. 

The center $\mathcal{Z}({\mathcal{C}})$ of ${\mathcal{C}}$ has a subcategory the unitary center $\mathcal{Z}^u({\mathcal{C}})$ of ${\mathcal{C}}$, such that the inclusion is an equivalence of categories. Theorem 17.1 in \cite{TuVi} implies that 
\[
|\cdot|_{\mathcal{C}} \simeq \tau_{\mathcal{Z}(\mathcal{C})} \simeq \tau_{\mathcal{Z}^{u}(\mathcal{C})}. 
\]
Since $\mathcal{Z}^u({\mathcal{C}})$ is an anomaly-free unitary modular category, these  TQFTs are unitary.

In this section we introduce the idea of the proof, following \cite{TVi}.

\subsubsection{Extending to graph TQFTs}

The goal of this subsection is Theorem 11.1 and Theorem 11.2 in \cite{TVi}:

\begin{thm}[\cite{TVi}]
 Let  $\mathcal{C}$  be a spherical fusion category over an algebraically closed field $\mathbb{k}$ such that $ \operatorname{dim} \mathcal{C} \neq 0 $. Then  $|M|_{\mathcal{C}}=\tau_{\mathcal{Z}(\mathcal{C})}(M)$  for any closed oriented $3$-manifold  $M $. Moreover  the TQFTs  $|\cdot|_{\mathcal{C}} $  and  $\tau_{\mathcal{Z}(\mathcal{C})}$  are isomorphic. 
\label{thmtqftisom}
\end{thm}

%
%
%
%

Let  $\mathcal{C}$  be a monoidal category. A \textit{half braiding}  is a pair  $(A, \sigma) $, where  $A \in \operatorname{Ob}(\mathcal{C})$  and
$
\sigma=\left\{\sigma_{X}: A \otimes X \rightarrow X \otimes A\right\}_{X \in \operatorname{Ob}(\mathcal{C})}
$
is a natural isomorphism such that
$
\sigma_{X \otimes Y}=\left(\mathrm{id}_{X} \otimes \sigma_{Y}\right)\left(\sigma_{X} \otimes \mathrm{id}_{Y}\right)
$
for all  $X, Y \in \mathrm{Ob}(\mathcal{C}) $. 

The \textit{center} of  $\mathcal{C}$  is the braided category  $\mathcal{Z}(\mathcal{C}) $ whose objects   are half braidings. A morphism  $(A, \sigma) \rightarrow\left(A^{\prime}, \sigma^{\prime}\right)$  in  $\mathcal{Z}(\mathcal{C})$  is a morphism  $f: A \rightarrow A^{\prime}$   such that  $\left(\mathrm{id}_{X} \otimes f\right) \sigma_{X}=\sigma_{X}^{\prime}\left(f \otimes \operatorname{id}_{X}\right)$  for all  $X \in \operatorname{Ob}(\mathcal{C}) $. There is a forgetful functor.

The following theorem is Theorem 1.2 and  Proposition 5.18 in
\cite{muger}. 

\begin{thm}[\cite{muger}]
Let  $\mathcal{C}$  be a spherical fusion $\mathbb{k}$-category such that  $\operatorname{dim} \mathcal{C} \neq 0 $. Then  $\mathcal{Z}(\mathcal{C})$  is an anomaly free modular category with $ \Delta_{+}=\Delta_{-}=\operatorname{dim} \mathcal{C} $, 
 $\operatorname{dim} \mathcal{Z}(\mathcal{C})=\Delta_{+} \Delta_{-}=(\operatorname{dim} \mathcal{C})^{2} $.
\end{thm}

Hence it makes sense to talk about the  Reshetikhin--Turaev extended TQFT. 

It is necessary to extend the definition of a TQFT. 
For any category  $\mathcal{B} $, we define a category  $\mathcal{L}_{\mathcal{B}}$  of $3 $-cobordisms with  $\mathcal{B} $-colored framed oriented links inside.  The links  may be empty so that the category  $\mathrm{Cob}_{3}$  is a subcategory of $ \mathcal{L}_{\mathcal{B}} $. 

By a \textit{link  TQFT}  we mean a symmetric monoidal functor  $Z: \mathcal{L}_{\mathcal{B}} \rightarrow  \Mod_{\mathbb{k}} $.

For a category  $\mathcal{B} $, we define a category  $\mathcal{G}_{\mathcal{B}} $  of $3$-cobordisms with  $\mathcal{B} $-colored ribbon graphs inside, see \cite{Tu}. Here we consider only ribbon graphs disjoint from the bases of cobordisms.  

 By a \textit{graph  TQFT}  we mean a symmetric monoidal functor  $\mathcal{G}_{\mathcal{B}} \rightarrow  \Mod_{\mathbb{k}} $.

The TQFT  $\tau_{\mathcal{Z}(\mathcal{C})}$  extends to a graph TQFT  still denoted  $\tau_{\mathcal{Z}(\mathcal{C})} $, see \cite{Tu}. This TQFT is \textit{non-degenerate}, i.e.  the vector space  $\tau_{\mathcal{Z}(\mathcal{C})}(\Sigma)$    is generated by the vectors  $\tau_{\mathcal{Z}(\mathcal{C})}(M, \emptyset, \Sigma)(1) $, where  $M$  runs over all compact oriented $3 $-manifolds with  $\mathcal{Z}(\mathcal{C}) $-colored ribbon graphs inside and with $ \partial M=\Sigma $. 

The TQFT  $|\cdot|_{\mathcal{C}}$  also can be extended to a graph TQFT. We discuss the extension to a link TQFT, which is similar.

Let $M$ be a compact oriented $3$-manifold, and $P$ be a skeleton. A link in $M$ can be presented by possibly intersecting circles in $P$ in general position, called the \textit{link diagram}. The underlying $4$-valent graph may be added to the $1$-skeleton of $P$. An \textit{enriched link diagram} in $P$ is a link diagram with half integers added to record the framing. 
Two enriched link diagrams in $P$ represent isotopic framed links if and only
if these diagrams may be related by a finite sequence of certain  moves and
ambient isotopies in $P$.

Consider a pair  $(M, L) $, where  $M$  is a compact oriented $3 $-manifold (with empty graph $G=\emptyset$ in the boundary) and  $L=L_1\sqcup \dots\sqcup L_N \subset \operatorname{Int}(M)$  is an oriented framed link whose all components are non-trivial and colored by simple objects $J_1\dots J_N$ of  $\mathcal{Z}(\mathcal{C}) $. We define   $|M, L|_{\mathcal{C}} \in \mathbb{k}$  as follows. 
Fix a
(finite) representative set $I$ of simple objects of $\mathcal{C}$. 
Pick a {special skeleton} (see \cite{TVi})  $P$  of  $M$  and an oriented enriched link diagram  $d $ in  $P$  representing  $L $. After adding $d$ to $P^{(1)}$, we obtain regions.

Pick distinguished square root  $\nu_{J_{q}} \in \mathbb{k}$  of the twist scalar  $v_{J_{q}}$  of  $J_{q} $. Let  $n_{q} \in \frac{1}{2} \mathbb{Z}$  be the pre-twist of the loop of  $d$  representing  $L_{q} $. Set
\[
|M, L|_{\mathcal{C}}=(\operatorname{dim}(\mathcal{C}))^{-|P|} \prod_{q=1}^{N} \nu_{J_{q}}^{2 n_{q}} \sum_{c}\left(\prod_{r \in \operatorname{Reg}(d)}(\operatorname{dim} c(r))^{\chi(r)}\right) *_{P}\left(\otimes_{x}|x|_{c}\right) . 
\]

This is a topological invariant of the pair $(M, L)$. It extends to the case where links may have trivial components.

It is shown in Theorem 1.2 in \cite{muger} that here $\mathcal{Z}(\mathcal{C}) $ is semisimple with finitely many simple objects. The 
invariant  extends by linearity to arbitrary colors of the components of $L$.

We  extend the  TQFT $|\cdot|_{\mathcal{C}}: \operatorname{Cob}_{3} \rightarrow \Mod_{\mathbb{k}}$ to a link TQFT  $\mathcal{L}_{\mathcal{Z}(\mathcal{C})} \rightarrow   \Mod_{\mathbb{k}} $, and moreover a graph TQFT.

\subsubsection{$|M, R|_{\mathcal{C}}=\tau_{\mathcal{Z}(\mathcal{C})}(M, R)$}

For a closed (connected) oriented  $3$-manifold $M$ with $\mathcal{Z}(\mathcal{C}) $-colored ribbon graph $R$ without free ends inside, we claim 

\begin{lem}
$|M, R|_{\mathcal{C}}=\tau_{\mathcal{Z}(\mathcal{C})}(M, R)$. 
\label{lemequalvalue}
\end{lem}


%
%
%
%

Present $M$ by surgery on $S^3$ along a framed oriented
link $L$. Pick a closed regular neighborhood $U$ of $L$, and set $E=S^3\backslash \operatorname{Int}(U)$. We consider $R\subset E$. 

Pick a homeomorphism
\[
f_{L}:\left(S^{1} \times S^{1}\right)^{\sqcup n} \longrightarrow \partial U=-\partial E
\]
which carries the  $q $-th copy of  $S^{1} \times\{\mathrm{pt}\}$  onto a positively oriented meridian of  $L_{q}$  and carries the  $q $-th copy of  $\{\mathrm{pt}\} \times S^{1}$  onto a positively oriented longitude of $ L_{q}$  determined by the framing.

Define $V=(-(S^1 \times D^2))^{\sqcup n} $, $W=(D^2\times S^1)^{\sqcup n}$. 
Using $f_{L}$, 
the gluing of $W$ to $E$ yields $S^3$, while the gluing of $V$ to $E$ yields $M$. 
The \textit{torus vector} of a TQFT  $Z$  is the vector
$
w=w(Z)=Z(V)\left(1_{\mathbb{k}}\right) \in Z(\partial V)=Z\left(S^{1} \times S^{1}\right) 
$
where $n=1$.

There are morphisms
\[
Z\left(f_{L}\right)=(f_{L})_\#: Z\left(\left(S^{1} \times S^{1}\right)^{\sqcup n}\right) \rightarrow Z(-\partial E) , 
\]
\[
Z(E, R): Z(-\partial E) \rightarrow \mathbb{k} .
\]

Define
$
Z_{R}^{L}
$
to be the composition. 
Hence
\[
Z\left(M, R\right)=Z_{R}^{L}\left(w^{\otimes n}\right). 
\]

We fix a representative set $\mathcal{J}$ of simple
objects of $\mathcal{Z}(\mathcal{C})$.

 For  $X \in \mathrm{Ob}(\mathcal{Z}(\mathcal{C})) $, let
\[
K_{X}=\{0\} \times S^{1} \subset \operatorname{Int}\left(D^{2} \times S^{1}\right)
\]
be the  $\mathcal{Z}(\mathcal{C})$-colored knot with orientation induced by that of $ S^{1} $, constant framing, and color  $X $. Set
\[
[X]=Z\left(D^{2} \times S^{1}, K_{X}\right)\left(1_{\mathbb{k}}\right) \in Z\left(S^{1} \times S^{1}\right). 
\]
Indeed, 
\begin{lem}
The vectors $\{[j]\}_{j \in \mathcal{J}}$   form a basis. 
\end{lem}

One can explicitly expand $w$ in terms of this basis. 
To calculate $|M, R|_{\mathcal{C}}$, we need $Z_{R}^{L}\left(w^{\otimes n}\right)$, and it suffices to calculate 
%
\[
Z_{R}^{L}\left(\left[j_{1}\right] \otimes_{\mathbb{k}} \cdots \otimes_{\mathbb{k}}\left[j_{n}\right]\right)=Z\left(S^{3}, T_{j_{1}, \ldots, j_{n}}\right)=(\operatorname{dim}(\mathcal{C}))^{-1}\left\langle T_{j_{1}, \ldots, j_{n}}\right\rangle_{\mathcal{Z}(\mathcal{C})}
\]
\[
 \in \operatorname{End}_{\mathcal{Z}(\mathcal{C})}\left(\mathbb{1}_{\mathcal{Z}(\mathcal{C})}\right)=\operatorname{End}_{\mathcal{C}}(\mathbb{1})=\mathbb{k}, 
\]
where $ j_{1}, \ldots, j_{n} \in \mathcal{J}$, $T_{j_{1}, \ldots, j_{n}}$  is  the union of  $R$  with the link  $L$  whose components  $L_{1}, \ldots, L_{n} $ are colored with  $j_{1}, \ldots, j_{n}$, and  
$\langle \cdot\rangle_{\mathcal{Z}(\mathcal{C})} $ is the graphical calculus, see Theorem 16.1 in \cite{TuVi}.

Now Lemma~\ref{lemequalvalue} follows from a complete expansion.


\subsubsection{Proof of Theorem~\ref{thmtqftisom}}

We claim that there is a monoidal isomorphism of the functors  $\tau_{\mathcal{Z}(\mathcal{C})}$  and  $|\cdot|_{\mathcal{C}}$  from  $\mathcal{G}_{\mathcal{Z}(\mathcal{C})}$  to $\Mod_{\mathbb{k}}$.  

Indeed there is  a general criterion establishing isomorphism of two (generalized)  TQFTs, see \cite{Tu}, Chapter III, Section 3.  If 

(1) at least one of the TQFTs is non-degenerate, 

(2)  the values of these TQFTs on cobordisms with empty bases are equal, and 

(3) the vector spaces associated  with any object have equal dimensions, 

then these TQFTs are isomorphic.  

The graph TQFT  $\tau_{\mathcal{Z}(\mathcal{C})} $  is non-degenerate. That $ |M, K|_{\mathcal{C}}=\tau_{\mathcal{Z}(\mathcal{C})}(M, K) $ for any  $\mathcal{Z}(\mathcal{C}) $-colored ribbon graph  $K$  in a closed oriented $3 $-manifold  $M$  is proven. The equality of dimensions is provided by the following lemma. This completes the proof of the theorem.

The following lemma is Lemma 11.3 in \cite{TVi}. 

\begin{lem}[\cite{TVi}]
The vector spaces  $|\Sigma|_{\mathcal{C}}$  and  $\tau_{\mathcal{Z}(\mathcal{C})}(\Sigma)$  associated with any closed connected oriented surface  $\Sigma$  have equal dimensions. 
\end{lem}

Its proof involves more tools like coends and Hopf monads, see \cite{TVi}.

\subsubsection{Proof of unitarity}


A half braiding $ (A, \sigma)$  of  $\mathcal{C} $ is  \textit{unitary} if $ \overline{\sigma_{X}}=\sigma_{X}^{-1}$  for all $ X \in \mathrm{Ob}(\mathcal{C}) $. The \textit{unitary center}  $\mathcal{Z}^{u}(\mathcal{C}) $  of  $\mathcal{C} $  is the full subcategory of $ \mathcal{Z}(\mathcal{C})$  formed by the unitary half braidings. 

The inclusion  $\mathcal{Z}^{u}(\mathcal{C}) \subset \mathcal{Z}(\mathcal{C}) $ is a braided equivalence, see  Theorem 6.4 in  \cite{muger}. Therefore $ \mathcal{Z}^{u}(\mathcal{C})$  is also  modular. The induced conjugation makes   $\mathcal{Z}^{u}(\mathcal{C}) $  a unitary modular category,  and hence  the TQFT  $\tau_{\mathcal{Z}^{u}(\mathcal{C})} $  is unitary. 

Now there are isomorphisms of TQFTs  $$\tau_{\mathcal{Z}^{u}}(\mathcal{C}) \simeq \tau_{\mathcal{Z}(\mathcal{C})} \simeq|\cdot|_{\mathcal{C}}. $$  The first one is induced by the braided equivalence  $\mathcal{Z}^{u}(\mathcal{C}) \simeq \mathcal{Z}(\mathcal{C})$. The unitary structure of $ \tau_{\mathcal{Z}^{u}(\mathcal{C})} $ is transported to $ |\cdot|_{\mathcal{C}} $ via the isomorphism.

By  Corollary 11.6, Theorem 11.2, and Theorem  11.5  of \cite{Tu}, 
\begin{cor}[\cite{TVi}]
If  $\mathcal{C}$  is a unitary fusion category, then  $$|| M|_{\mathcal{C}} | \leq(\operatorname{dim}(\mathcal{C}))^{g(M)-1}$$  for any closed oriented $3$-manifold  $M $, where  $g(M)$  is the Heegaard genus of  $M $.
\end{cor}

In the next chapter we give another proof of a similar theorem.


\section{Counterexamples to Rank versus Genus Conjecture}

\subsection{Estimating the Heegaard genus}



In this section we estimate the Heegaard genus, using arguments similar to Theorem 2.2 in \cite{garoufalidis}. {This theorem is not correct in general. }Note that Theorem 2.2 uses Lemma 2.1, and  the author claims that Lemma 2.1 follows from a paper by Witten \cite{witten}. In \cite{witten}, there is an implicit condition on the TQFT, namely the space associated to $S^2$ must be one-dimensional. This condition is not mentioned in \cite{garoufalidis}. 

However, the estimate is valid in our special case. 

\begin{lem}[Lemma 13.6 in \cite{TuVi}]
Over any commutative
ring $\mathbb{k}$, the $\mathbb{k}$-module $|S^2|_{\mathcal{C}}$ is isomorphic to $\mathbb{k}$. 
\end{lem}

The proof is just a direct calculation using definitions.

\begin{lem}[Lemma 10.2 in \cite{TuVi}]
Let  $Z: \operatorname{Cob}_{n} \rightarrow \operatorname{Mod}_{\mathbb{k}}$  be an  $n $-dimensional TQFT such that $Z\left(S^{n-1}\right) \simeq \mathbb{k}$. Then, for any connected sum  $M=M_{0} \# M_{1}$, we have 
\[
Z\left(S^{n}\right) Z\left(M,-\partial M_{0}, \partial M_{1}\right)=Z\left(M_{1}, \emptyset, \partial M_{1}\right) \circ Z\left(M_{0},-\partial M_{0}, \emptyset\right). 
\]
In particular, if  $\partial M_{0}=\partial M_{1}=\emptyset $, then  $Z\left(S^{n}\right) Z(M)=Z\left(M_{0}\right) Z\left(M_{1}\right) $. 
\end{lem}

\begin{proof}[Sketch of proof]
For  $i=0,1 $, pick a ball   $B_{i} \subset \operatorname{Int}\left(M_{i}\right) $ and set  $N_{i}=M_{i} \backslash \operatorname{Int}\left(B_{i}\right)$. 
Let
\[
b_{0}: S^{n-1} \rightarrow \emptyset, \quad b_{1}: \emptyset \rightarrow S^{n-1}, \quad n_{0}:-\partial M_{0} \rightarrow S^{n-1}, \quad n_{1}: S^{n-1} \rightarrow \partial M_{1}
\]
be the corresponding morphisms. 

Pick a  $\mathbb{k} $-linear isomorphism  $z: Z\left(S^{n-1}\right) \rightarrow   Z(\emptyset) $. Then $ Z\left(b_{0}\right)=\lambda_{0} z $ and  $Z\left(b_{1}\right)=\lambda_{1} z^{-1}$  for some  $\lambda_{0}, \lambda_{1} \in \mathbb{k}$. It follows  that
\[
Z\left(S^{n}\right) \operatorname{id}_{Z(\emptyset)}=Z\left(S^{n}, \emptyset, \emptyset\right)=Z\left(b_{0} \circ b_{1}\right)=Z\left(b_{0}\right) Z\left(b_{1}\right)=\lambda_{0} \lambda_{1} \mathrm{id}_{Z(\emptyset)}. 
\]
Hence
\[
\begin{array}{l}
Z\left(S^{n}\right) Z\left(M,-\partial M_{0}, \partial M_{1}\right) \\
=\lambda_{0} \lambda_{1} Z\left(n_{1} \circ n_{0}\right) \\
=\lambda_{0} \lambda_{1} Z\left(n_{1}\right) \circ Z\left(n_{0}\right) \\
=Z\left(n_{1}\right) \circ\left(\lambda_{1} z^{-1}\right) \circ\left(\lambda_{0} z\right) \circ Z\left(n_{0}\right) \\
=Z\left(n_{1} \circ b_{1}\right) \circ Z\left(b_{0} \circ n_{0}\right)\\
=Z\left(M_{1}, \emptyset, \partial M_{1}\right) \circ Z\left(M_{0},-\partial M_{0}, \emptyset\right) .
\end{array}
\]
\end{proof}

A direct calculation in  section 13.1.3 in \cite{TuVi} shows that 

\begin{lem}[\cite{TuVi}]
$|S^1\times S^2|_{{\mathcal{C}}}=1$, and   $|S^{3}|_{{\mathcal{C}}}=(\operatorname{dim}({\mathcal{C}}))^{-1}$.  
\end{lem}

Here, we consider the original Turaev--Viro TQFT \cite{TV}, at $q=e^{\pi i/r}$, $r\geq 3$. Then 
\[
|S^3|=-\frac{\left(e^{\frac{i \pi }{r}}-e^{-\frac{i \pi }{r}}\right)^2}{2 r}=\frac{2 \sin ^2\left(\frac{\pi }{r}\right)}{r}\in (0,1). 
\]

\begin{lem}
Let $M,N$ be  compact oriented $3$-manifolds and $f:\partial M\to \partial N$ be an orientation-preserving homeomorphism. Let $Z$ be a unitary TQFT, and write $f_{\#}:Z(\partial M)\to Z(\partial N)$ for the isomorphism induced by $f$. Then 
\[
Z(M\cup_f-N)=\langle f_{\#}Z(M),Z(N)\rangle_{Z(\partial N)}\in \mathbb{C}. 
\]
\end{lem}

\proof{In $\cob$ there are morphisms
\[
-N:\partial N\to \emptyset, \quad N:\emptyset \to \partial N. 
\]
Hence
\[
Z(-N):Z(\partial N)\to Z(\emptyset)=\mathbb{C}, \quad Z(N):\mathbb{C} \to Z(\partial N). 
\]

Note that $N$ and $-N$ are conjugate. By the definition of unitary TQFT, 
\[
\begin{array}{l}
\langle f_{\#}Z(M),Z(N)\rangle_{Z(\partial N)}\\
=\langle f_{\#}Z(M)1,Z(N)1\rangle_{Z(\partial N)}\\
=\langle Z(-N)f_{\#}Z(M)1,1\rangle\\
=Z(M\cup_f-N). 
\end{array}
\]
\qedhere}

In the case where we glue two components $\Sigma_1$, $\Sigma_2$ of the boundary of one manifold $M$, it is equivalent to glue $\Sigma_1\times I$ to $M$, and the lemma applies as well.

\begin{thm}
Let $Z$ be the original Turaev--Viro TQFT \cite{TV}, at $q=e^{\pi i/r}$, $r\geq 3$. Let $M$ be a closed oriented $3$-manifold. Then 
\[
|Z(M)| \leq Z\left(S^{3}\right)^{-g(M)+1}, 
\]
where  $g(M)$  is the Heegaard genus of  $M$. Thus
\[
g(M)-1 \geq -\frac{\log |Z(M)|}{\log Z\left(S^{3}\right)}. 
\]
\end{thm}

\proof{
Let $M=H\cup_f(-H)$ be a Heegaard splitting of $M$, where $H$ is a handlebody of genus $g$, and $f: \partial H\to \partial H$ is an orientation-preserving homeomorphism. Let $u:= Z(H)\in Z(\partial H)$. Then
\[
\begin{aligned}
|Z(M)| &=\left|\left\langle u, f_{\#}(u)\right\rangle\right| \\
& \leq \sqrt{\langle u, u\rangle\left\langle f_{\#}(u), f_{\#}(u)\right\rangle} \\
&=\langle u, u\rangle \\
&=Z\left(\#_{i=1}^{g} S^{2} \times S^{1}\right) \\
&=Z\left(S^{2} \times S^{1}\right)^{g} Z\left(S^{3}\right)^{-g+1} \\
&=Z\left(S^{3}\right)^{-g+1}. 
\end{aligned}
\]
\qedhere}

\subsection{Searching for counterexamples}
%
%
%
%
%
%

First we claim that 
\begin{prop}
There are no counterexamples to the rank versus genus conjecture among manifolds with  Heegaard genus  $0,1,2$. 
\end{prop}

\proof{

If $\pi_1(M)$ is a finite cyclic group, by elliptisation theorem $M$ is elliptic. The only elliptic $3$-manifolds with finite cyclic fundamental group are the lens spaces (and $S^3$). 

If $\pi_1(M)=\mathbb{Z}$, $M$ is irreducible and does  not contain incompressible torus and hence is geometric. $M$ has $S^2\times{\mathbb{R}}$ geometry, and is hence either $S^1\times S^2$ or $\mathbb{R}P^2\widetilde{\times }S^1$. $\pi_1(S^1\times S^2)=\mathbb{Z}$, while $\pi_1(\mathbb{R}P^2\widetilde{\times }S^1)$ is not $\mathbb{Z}$. Thus $M=S^1\times S^2$. 

\qedhere}

For the theorems used in the proof, see \cite{martelli}. 
Manifolds we are looking for have Heegaard genus at least $3$.

The software \textit{Regina} \cite{regina} is used in concrete calculation. Given a triangulation, \textit{Regina} computes the state sum invariant. 
The function \texttt{turaevViroApprox()} in \textit{Regina} computes the given Turaev--Viro state sum invariant of a $3$-manifold using a fast but inexact floating-point approximation. 
The function \texttt{turaevViro()} in \textit{Regina} computes the given Turaev--Viro state sum invariant of a $3$-manifold using exact arithmetic. 

We use  the census of all minimal triangulations of all closed prime orientable $3$-manifolds that can be built from $\leq 11$ tetrahedra, as tabulated using \textit{Regina} \cite{burton}. 


We use the original Turaev--Viro TQFT \cite{TV}, at $q=e^{\pi i/r}$, $r=5$.

%

\subsubsection{Notation}

Notation used in naming $3$-manifolds is described in \textit{Regina}.

The name   {T}  {x}  {I} /[a, b | c, d]  denotes a torus bundle over the circle. This torus bundle is expressed as the product of the torus and the interval, with the two torus boundaries identified according to the monodromy  $\left[\begin{array}{ll}a & b \\ c & d\end{array}\right] $.

A Seifert fibred space is denoted by SFS [B: ($p_{1}, q_{1}$)  ($p_{2}, q_{2}$) $\ldots$  ($p_{k}, q_{k}$)], where  B  is the base orbifold  and each   ($p_{i}, q_{i} $)  describes an exceptional fibre.

\begin{table}
  \centering
  \begin{tabular}{ll}
    \toprule
Symbol     &   Base orbifold\\
\midrule
S2         &  $2$-sphere \\
RP2/n2     &  Real projective plane\\
T          &  Torus \\
KB/n2      &  Klein bottle\\
D          &  Disc with regular boundary\\
M/n2      &   M\"obius band with regular boundary\\
A         &   Annulus with two regular boundaries\\
 \bottomrule
  \end{tabular}
\end{table}

In the special case where there are no exceptional fibres at all, we simply write the space as  B x S1  (if there are no fibre-reversing twists) or  B x\~{} S1  (if there are).

The name SFS   [$B_{1}$: $\ldots$]  {U}/{m} {SFS} [$B_{2}$: $\ldots$],  {m} = [ a,b | c,d  ]  denotes a graph manifold formed from a pair of Seifert fibred spaces $\mathcal{M}_{1}$, $\mathcal{M}_{2}$.  Let  $\phi_{i}$  and  $\omega_{i}$  be curves on the boundary of  $\mathcal{M}_{i}$  representing the fibres and base orbifold, $i=1,2$. We identify both boundaries so that
\[
\left[\begin{array}{l}
\phi_{2} \\
\omega_{2}
\end{array}\right]=\left[\begin{array}{ll}
a & b \\
c & d
\end{array}\right] \times\left[\begin{array}{l}
\phi_{1} \\
\omega_{1}
\end{array}\right]. 
\]

$3$-piece graph manifolds are defined similarly.

\subsubsection{Data for $r=5$}

Here is a table of manifolds with Turaev--Viro invariant   $\geq 7.235$. Their $H_1$ are shown. Here  \textcolor{red}{$\bigstar$} means the Heegaard genus is larger than the rank of $H_1$.

SFS [S2: (2,1) (2,1) (2,1) (2,-1)] : \#1: 13.105572809000083

2 Z\_2 + Z\_4

T x S1 : \#1: 15.999999999999984

3 Z

KB/n2 x\~{} S1 : \#1: 15.999999999999988

Z + 2 Z\_2

SFS [S2: (2,1) (2,1) (2,1) (2,1)] : \#1: 14.894427190999902

2 Z\_2 + Z\_8

SFS [S2: (2,1) (2,1) (2,1) (3,-2)] : \#1: 7.999999999999994

Z\_2 + Z\_10 \textcolor{red}{$\bigstar$} \textcolor{blue}{True by \cite{BZ}. }

SFS [S2: (2,1) (2,1) (2,1) (3,-1)] : \#1: 7.447213595499956

Z\_2 + Z\_14 \textcolor{red}{$\bigstar$} \textcolor{blue}{True by \cite{BZ}. }

SFS [S2: (2,1) (2,1) (2,1) (2,3)] : \#1: 14.894427190999908

2 Z\_2 + Z\_12

SFS [S2: (2,1) (2,1) (2,1) (3,2)] : \#1: 7.4472135954999485

Z\_2 + Z\_26 \textcolor{red}{$\bigstar$} \textcolor{blue}{True by \cite{BZ}. }

SFS [D: (2,1) (2,1)] U/m SFS [D: (2,1) (2,1)], m = [ -1,3 | 0,1 ] : \#1: 7.788854381999807

2 Z\_4 \textcolor{red}{$\bigstar$}

SFS [S2: (2,1) (2,1) (2,1) (2,5)] : \#1: 13.105572809000096

2 Z\_2 + Z\_16

SFS [S2: (2,1) (2,1) (2,1) (3,4)] : \#1: 7.44721359549994

Z\_2 + Z\_34 \textcolor{red}{$\bigstar$} \textcolor{blue}{True by \cite{BZ}. }

SFS [S2: (2,1) (2,1) (2,1) (5,-9)] : \#1: 9.919349550499504

Z\_2 + Z\_6 \textcolor{red}{$\bigstar$} \textcolor{blue}{True by \cite{BZ}. }

SFS [S2: (2,1) (2,1) (2,1) (5,-4)] : \#1: 9.919349550499517

Z\_2 + Z\_14 \textcolor{red}{$\bigstar$} \textcolor{blue}{True by \cite{BZ}. }

SFS [S2: (2,1) (2,1) (2,1) (5,-1)] : \#1: 9.919349550499518

Z\_2 + Z\_26 \textcolor{red}{$\bigstar$} \textcolor{blue}{True by \cite{BZ}. }

SFS [S2: (2,1) (2,1) (2,1) (7,-12)] : \#1: 7.4472135954999406

Z\_2 + Z\_6 \textcolor{red}{$\bigstar$} \textcolor{blue}{True by \cite{BZ}. }

SFS [S2: (2,1) (2,1) (2,1) (7,-4)] : \#1: 7.447213595499941

Z\_2 + Z\_26 \textcolor{red}{$\bigstar$} \textcolor{blue}{True by \cite{BZ}. }

SFS [S2: (2,1) (2,1) (2,1) (7,-3)] : \#1: 7.999999999999984

Z\_2 + Z\_30 \textcolor{red}{$\bigstar$} \textcolor{blue}{True by \cite{BZ}. }

SFS [S2: (2,1) (2,1) (2,1) (7,-2)] : \#1: 7.447213595499941

Z\_2 + Z\_34 \textcolor{red}{$\bigstar$} \textcolor{blue}{True by \cite{BZ}. }

SFS [S2: (2,1) (2,1) (2,1) (2,1) (2,-3)] : \#1: 74.2111456179999
\textcolor{blue}{Note!}

4 Z\_2

SFS [S2: (2,1) (2,1) (2,1) (2,1) (2,-1)] : \#1: 77.78885438199974
\textcolor{blue}{Note!}

3 Z\_2 + Z\_6

SFS [RP2/n2: (2,1) (2,1) (2,-1)] : \#1: 18.89442719099984

2 Z\_2 + Z\_8

SFS [RP2/n2: (2,1) (2,1) (2,1)] : \#1: 22.894427190999874

2 Z\_2 + Z\_8

SFS [T: (2,3)] : \#1: 9.236067977499754

2 Z + Z\_3

SFS [KB/n2: (2,3)] : \#1: 9.236067977499763

Z + Z\_8 \textcolor{red}{$\bigstar$} \textcolor{blue}{False. }

SFS [D: (2,1) (2,1)] U/m SFS [D: (2,1) (2,1)], m = [ -1,4 | 0,1 ] : \#1: 8.422291236000305

2 Z\_2 + Z\_8

SFS [D: (2,1) (2,1)] U/m SFS [D: (2,1) (2,1) (2,1)], m = [ 0,1 | 1,0 ] : \#1: 8.161300899000896

2 Z\_2 + Z\_4

SFS [S2: (2,1) (2,1) (2,1) (2,7)] : \#1: 15.999999999999968

2 Z\_2 + Z\_20

SFS [S2: (2,1) (2,1) (2,1) (3,7)] : \#1: 7.447213595499941

Z\_2 + Z\_46 \textcolor{red}{$\bigstar$} \textcolor{blue}{True by \cite{BZ}. }

SFS [S2: (2,1) (2,1) (2,1) (3,8)] : \#1: 7.999999999999971

Z\_2 + Z\_50 \textcolor{red}{$\bigstar$} \textcolor{blue}{True by \cite{BZ}. }

SFS [S2: (2,1) (2,1) (2,1) (5,1)] : \#1: 9.919349550499517

Z\_2 + Z\_34 \textcolor{red}{$\bigstar$} \textcolor{blue}{True by \cite{BZ}. }

SFS [S2: (2,1) (2,1) (2,1) (5,4)] : \#1: 9.919349550499517

Z\_2 + Z\_46 \textcolor{red}{$\bigstar$} \textcolor{blue}{True by \cite{BZ}. }

SFS [S2: (2,1) (2,1) (2,1) (7,2)] : \#1: 7.999999999999983

Z\_2 + Z\_50 \textcolor{red}{$\bigstar$} \textcolor{blue}{True by \cite{BZ}. }

SFS [S2: (2,1) (2,1) (2,1) (7,3)] : \#1: 7.447213595499948

Z\_2 + Z\_54  \textcolor{red}{$\bigstar$} \textcolor{blue}{True by \cite{BZ}. }

SFS [S2: (2,1) (2,1) (2,1) (13,-21)] : \#1: 7.447213595499955

Z\_2 + Z\_6 \textcolor{red}{$\bigstar$} \textcolor{blue}{True by \cite{BZ}. }

SFS [S2: (2,1) (2,1) (2,1) (13,-8)] : \#1: 7.447213595499948

Z\_2 + Z\_46 \textcolor{red}{$\bigstar$} \textcolor{blue}{True by \cite{BZ}. }

SFS [S2: (2,1) (2,1) (2,1) (2,1) (2,1)] : \#1: 71.99999999999923
\textcolor{blue}{Note!}

3 Z\_2 + Z\_10

SFS [S2: (2,1) (2,1) (2,1) (2,1) (3,-5)] : \#1: 37.10557280900008

2 Z\_2 + Z\_4

SFS [S2: (2,1) (2,1) (2,1) (2,1) (3,-2)] : \#1: 37.10557280900008

2 Z\_2 + Z\_16

SFS [S2: (2,1) (2,1) (2,1) (2,1) (3,-4)] : \#1: 38.894427190999856

2 Z\_2 + Z\_8

SFS [S2: (2,1) (2,1) (2,1) (2,1) (3,-1)] : \#1: 36.00000000000012

2 Z\_2 + Z\_20

SFS [RP2/n2: (2,1) (2,1) (2,3)] : \#1: 16.42229123600026

2 Z\_2 + Z\_8

SFS [RP2/n2: (2,1) (2,1) (3,-2)] : \#1: 9.447213595499948

Z\_4 + Z\_12 \textcolor{red}{$\bigstar$}\textcolor{blue}{False. }

SFS [RP2/n2: (2,1) (2,1) (3,1)] : \#1: 9.447213595499948

Z\_4 + Z\_12 \textcolor{red}{$\bigstar$}\textcolor{blue}{False. }

SFS [RP2/n2: (2,1) (2,1) (3,-1)] : \#1: 11.447213595499925

Z\_4 + Z\_12 \textcolor{red}{$\bigstar$}\textcolor{blue}{False. }

SFS [RP2/n2: (2,1) (2,1) (3,2)] : \#1: 8.211145618000165

Z\_4 + Z\_12 \textcolor{red}{$\bigstar$}\textcolor{blue}{False. }

SFS [T: (1,5)] : \#1: 7.999999999999971

2 Z + Z\_5

SFS [T: (3,4)] : \#1: 9.527864045000376

2 Z + Z\_4

SFS [KB/n2: (1,5)] : \#1: 7.999999999999987

Z + Z\_4 \textcolor{red}{$\bigstar$}\textcolor{blue}{False. }

SFS [KB/n2: (3,4)] : \#1: 9.52786404500038

Z + Z\_2 + Z\_6

SFS [D: (2,1) (2,1)] U/m SFS [D: (2,1) (2,1)], m = [ 1,4 | 0,1 ] : \#1: 9.105572809000032

2 Z\_2 + Z\_16

SFS [D: (2,1) (2,1)] U/m SFS [D: (2,1) (2,1)], m = [ 4,3 | 1,1 ] : \#1: 7.2360679774997605

Z\_4 + Z\_20

SFS [D: (2,1) (2,1)] U/m SFS [D: (2,1) (2,1)], m = [ -1,5 | 0,1 ] : \#1: 7.44721359549993

Z\_4 + Z\_12 \textcolor{red}{$\bigstar$}

SFS [D: (2,1) (2,1)] U/m SFS [D: (2,1) (2,1)], m = [ 5,3 | 2,1 ] : \#1: 7.236067977499749

Z\_4 + Z\_20

SFS [D: (2,1) (2,1)] U/m SFS [D: (2,1) (2,1)], m = [ -5,7 | -3,4 ] : \#1: 7.788854381999794

2 Z\_4 \textcolor{red}{$\bigstar$}

SFS [D: (2,1) (2,1)] U/m SFS [D: (2,1) (2,1)], m = [ -5,8 | -3,5 ] : \#1: 15.577708763999555

2 Z\_2 + Z\_4

SFS [D: (2,1) (2,1)] U/m SFS [D: (2,1) (2,1)], m = [ -5,8 | -2,3 ] : \#1: 10.89442719099983

2 Z\_2 + Z\_8

SFS [D: (2,1) (2,1)] U/m SFS [D: (3,1) (3,2)], m = [ -1,3 | 0,1 ] : \#1: 7.788854381999803

Z\_2 + Z\_18 \textcolor{red}{$\bigstar$}

SFS [D: (2,1) (2,1)] U/m SFS [D: (3,2) (3,2)], m = [ -1,3 | -1,2 ] : \#1: 7.236067977499761

Z\_2 + Z\_30

SFS [D: (2,1) (2,1)] U/m SFS [D: (2,1) (2,1) (2,1)], m = [ 1,1 | 0,1 ] : \#1: 12.1613008990009

2 Z\_2 + Z\_16

SFS [D: (2,1) (2,1)] U/m SFS [D: (2,1) (2,1) (2,1)], m = [ 1,1 | 1,0 ] : \#1: 16.161300899000842

2 Z\_2 + Z\_16

SFS [D: (2,1) (2,1)] U/m SFS [D: (2,1) (2,1) (2,1)], m = [ 1,1 | 1,2 ] : \#1: 15.05572809000077

Z + 2 Z\_2

SFS [D: (2,1) (3,1)] U/m SFS [D: (2,1) (2,1) (2,1)], m = [ 0,1 | 1,0 ] : \#1: 11.447213595499912

Z\_2 + Z\_6 \textcolor{red}{$\bigstar$}

SFS [D: (2,1) (3,1)] U/m SFS [D: (2,1) (2,1) (2,1)], m = [ 0,1 | 1,1 ] : \#1: 9.447213595499933

Z\_2 + Z\_14 \textcolor{red}{$\bigstar$}

SFS [D: (2,1) (3,1)] U/m SFS [D: (2,1) (2,1) (2,1)], m = [ 0,1 | -1,1 ] : \#1: 9.44721359549994

Z\_2 + Z\_34 \textcolor{red}{$\bigstar$}

SFS [D: (2,1) (3,2)] U/m SFS [D: (2,1) (2,1) (2,1)], m = [ 0,1 | 1,0 ] : \#1: 8.080650449500448

Z\_2 + Z\_18 \textcolor{red}{$\bigstar$}

SFS [D: (2,1) (3,2)] U/m SFS [D: (2,1) (2,1) (2,1)], m = [ 0,1 | 1,1 ] : \#1: 7.527864045000408

Z\_2 + Z\_10 \textcolor{red}{$\bigstar$}

SFS [S2: (2,1) (2,1) (2,1) (2,9)] : \#1: 13.105572809000092

2 Z\_2 + Z\_24

SFS [S2: (2,1) (2,1) (2,1) (5,6)] : \#1: 9.919349550499511

Z\_2 + Z\_54 \textcolor{red}{$\bigstar$} \textcolor{blue}{True by \cite{BZ}. }

SFS [S2: (2,1) (2,1) (2,1) (5,9)] : \#1: 9.919349550499511

Z\_2 + Z\_66 \textcolor{red}{$\bigstar$} \textcolor{blue}{True by \cite{BZ}. }

SFS [S2: (2,1) (2,1) (2,1) (7,-13)] : \#1: 7.999999999999955

Z\_2 + Z\_10 \textcolor{red}{$\bigstar$} \textcolor{blue}{True by \cite{BZ}. }

SFS [S2: (2,1) (2,1) (2,1) (7,11)] : \#1: 7.447213595499942

Z\_2 + Z\_86 \textcolor{red}{$\bigstar$} \textcolor{blue}{True by \cite{BZ}. }

SFS [S2: (2,1) (2,1) (2,1) (7,12)] : \#1: 7.999999999999976

Z\_2 + Z\_90 \textcolor{red}{$\bigstar$} \textcolor{blue}{True by \cite{BZ}. }

SFS [S2: (2,1) (2,1) (2,1) (13,-23)] : \#1: 7.447213595499916

Z\_2 + Z\_14 \textcolor{red}{$\bigstar$} \textcolor{blue}{True by \cite{BZ}. }

SFS [S2: (2,1) (2,1) (2,1) (13,-22)] : \#1: 7.999999999999953

Z\_2 + Z\_10 \textcolor{red}{$\bigstar$} \textcolor{blue}{True by \cite{BZ}. }

SFS [S2: (2,1) (2,1) (2,1) (13,8)] : \#1: 7.999999999999988

Z\_2 + Z\_110 \textcolor{red}{$\bigstar$} \textcolor{blue}{True by \cite{BZ}. }

SFS [S2: (2,1) (2,1) (2,1) (13,-3)] : \#1: 7.447213595499927

Z\_2 + Z\_66 \textcolor{red}{$\bigstar$} \textcolor{blue}{True by \cite{BZ}. }

SFS [S2: (2,1) (2,1) (2,1) (15,-26)] : \#1: 9.919349550499476

Z\_2 + Z\_14 \textcolor{red}{$\bigstar$} \textcolor{blue}{True by \cite{BZ}. }

SFS [S2: (2,1) (2,1) (2,1) (15,-11)] : \#1: 9.919349550499486

Z\_2 + Z\_46 \textcolor{red}{$\bigstar$} \textcolor{blue}{True by \cite{BZ}. }

SFS [S2: (2,1) (2,1) (2,1) (15,-4)] : \#1: 9.919349550499486

Z\_2 + Z\_74 \textcolor{red}{$\bigstar$} \textcolor{blue}{True by \cite{BZ}. }

SFS [S2: (2,1) (2,1) (2,1) (17,-29)] : \#1: 7.447213595499921

Z\_2 + Z\_14 \textcolor{red}{$\bigstar$} \textcolor{blue}{True by \cite{BZ}. }

SFS [S2: (2,1) (2,1) (2,1) (17,-12)] : \#1: 7.447213595499927

Z\_2 + Z\_54 \textcolor{red}{$\bigstar$} \textcolor{blue}{True by \cite{BZ}. }

SFS [S2: (2,1) (2,1) (2,1) (17,-27)] : \#1: 7.4472135954999255

Z\_2 + Z\_6 \textcolor{red}{$\bigstar$} \textcolor{blue}{True by \cite{BZ}. }

SFS [S2: (2,1) (2,1) (2,1) (17,-7)] : \#1: 7.447213595499917

Z\_2 + Z\_74 \textcolor{red}{$\bigstar$} \textcolor{blue}{True by \cite{BZ}. }

SFS [S2: (2,1) (2,1) (2,1) (18,-31)] : \#1: 14.894427190999846

2 Z\_2 + Z\_8

SFS [S2: (2,1) (2,1) (2,1) (18,-13)] : \#1: 14.894427190999856

2 Z\_2 + Z\_28

SFS [S2: (2,1) (2,1) (2,1) (18,-29)] : \#1: 13.10557280900005

2 Z\_2 + Z\_4

SFS [S2: (2,1) (2,1) (2,1) (18,-11)] : \#1: 14.894427190999844

2 Z\_2 + Z\_32

SFS [S2: (2,1) (2,1) (2,1) (18,-7)] : \#1: 15.999999999999904

2 Z\_2 + Z\_40

SFS [S2: (2,1) (2,1) (2,1) (18,-5)] : \#1: 13.10557280900005

2 Z\_2 + Z\_44

SFS [S2: (2,1) (2,1) (2,1) (2,1) (2,3)] : \#1: 77.788854382
\textcolor{blue}{Note!}

3 Z\_2 + Z\_14

SFS [S2: (2,1) (2,1) (2,1) (2,1) (3,1)] : \#1: 38.89442719099981

2 Z\_2 + Z\_28

SFS [S2: (2,1) (2,1) (2,1) (2,1) (3,2)] : \#1: 38.89442719099981

2 Z\_2 + Z\_32

SFS [S2: (2,1) (2,1) (2,1) (3,1) (3,-5)] : \#1: 19.447213595499942

Z\_2 + Z\_6 \textcolor{red}{$\bigstar$}\textcolor{blue}{False. }

SFS [S2: (2,1) (2,1) (2,1) (3,1) (3,-2)] : \#1: 18.552786404500043

Z\_2 + Z\_42 \textcolor{red}{$\bigstar$}\textcolor{blue}{False. }

SFS [S2: (2,1) (2,1) (2,1) (3,1) (3,-4)] : \#1: 18.55278640450001

Z\_2 + Z\_18 \textcolor{red}{$\bigstar$}\textcolor{blue}{False. }

SFS [S2: (2,1) (2,1) (2,1) (3,1) (3,-1)] : \#1: 19.4472135955

Z\_2 + Z\_54 \textcolor{red}{$\bigstar$}\textcolor{blue}{False. }

SFS [S2: (2,1) (2,1) (2,1) (3,2) (3,-4)] : \#1: 17.999999999999954

Z\_2 + Z\_30 \textcolor{red}{$\bigstar$}\textcolor{blue}{False. }

SFS [S2: (2,1) (2,1) (2,1) (3,2) (3,-1)] : \#1: 19.447213595500003

Z\_2 + Z\_66 \textcolor{red}{$\bigstar$}\textcolor{blue}{False. }

SFS [RP2/n2: (2,1) (2,1) (2,5)] : \#1: 22.894427191000126

2 Z\_2 + Z\_8

SFS [RP2/n2: (2,1) (2,1) (3,4)] : \#1: 11.447213595499914

Z\_4 + Z\_12 \textcolor{red}{$\bigstar$}\textcolor{blue}{False. }

SFS [RP2/n2: (2,1) (2,1) (3,5)] : \#1: 11.447213595499916

Z\_4 + Z\_12 \textcolor{red}{$\bigstar$}\textcolor{blue}{False. }

SFS [T: (2,7)] : \#1: 9.236067977499735

2 Z + Z\_7

SFS [T: (3,7)] : \#1: 9.236067977499753

2 Z + Z\_7

SFS [T: (3,8)] : \#1: 18.472135954999473

2 Z + Z\_8

SFS [T: (4,5)] : \#1: 7.9999999999999645

2 Z + Z\_5

SFS [KB/n2: (2,7)] : \#1: 9.236067977499738

Z + Z\_8 \textcolor{red}{$\bigstar$}\textcolor{blue}{False. }

SFS [KB/n2: (3,7)] : \#1: 9.23606797749975

Z + Z\_12 \textcolor{red}{$\bigstar$}\textcolor{blue}{False. }

SFS [KB/n2: (3,8)] : \#1: 18.472135954999477

Z + Z\_2 + Z\_6

SFS [KB/n2: (4,5)] : \#1: 7.999999999999974

Z + Z\_16 \textcolor{red}{$\bigstar$}\textcolor{blue}{False. }

T x I / [ 13,8 | 8,5 ] : \#1: 10.472135954999501

Z + 2 Z\_4

T x I / [ -13,-8 | -8,-5 ] : \#1: 10.472135954999501

Z + Z\_2 + Z\_10

SFS [D: (2,1) (2,1)] U/m SFS [D: (2,1) (2,1)], m = [ 1,4 | 0,-1 ] : \#1: 11.577708763999615

2 Z\_2 + Z\_24

SFS [D: (2,1) (2,1)] U/m SFS [D: (2,1) (2,1)], m = [ -3,-4 | 1,1 ] : \#1: 9.105572809000037

2 Z\_2 + Z\_36

SFS [D: (2,1) (2,1)] U/m SFS [D: (2,1) (2,1)], m = [ -3,-4 | 2,3 ] : \#1: 8.422291236000294

2 Z\_2 + Z\_48

SFS [D: (2,1) (2,1)] U/m SFS [D: (2,1) (2,1)], m = [ 1,5 | 0,1 ] : \#1: 7.999999999999984

Z\_4 + Z\_20 \textcolor{red}{$\bigstar$}

SFS [D: (2,1) (2,1)] U/m SFS [D: (2,1) (2,1)], m = [ 4,5 | 1,1 ] : \#1: 7.447213595499944

Z\_4 + Z\_28 \textcolor{red}{$\bigstar$}

SFS [D: (2,1) (2,1)] U/m SFS [D: (2,1) (2,1)], m = [ 4,5 | 3,4 ] : \#1: 7.447213595499939

Z\_4 + Z\_8 \textcolor{red}{$\bigstar$}

SFS [D: (2,1) (2,1)] U/m SFS [D: (2,1) (2,1)], m = [ 5,7 | 3,4 ] : \#1: 7.2360679774997685

Z\_4 + Z\_20

SFS [D: (2,1) (2,1)] U/m SFS [D: (2,1) (2,1)], m = [ 3,8 | 2,5 ] : \#1: 11.57770876399964

2 Z\_2 + Z\_16

SFS [D: (2,1) (2,1)] U/m SFS [D: (2,1) (2,1)], m = [ 5,8 | 2,3 ] : \#1: 10.894427190999833

2 Z\_2 + Z\_32

SFS [D: (2,1) (2,1)] U/m SFS [D: (2,1) (2,1)], m = [ 5,8 | 3,5 ] : \#1: 14.472135954999498

2 Z\_2 + Z\_20

SFS [D: (2,1) (2,1)] U/m SFS [D: (2,1) (2,1)], m = [ -7,12 | -3,5 ] : \#1: 10.894427190999867

2 Z\_2 + Z\_12

SFS [D: (2,1) (2,1)] U/m SFS [D: (2,1) (2,1)], m = [ -5,12 | -2,5 ] : \#1: 15.577708763999548

2 Z\_2 + Z\_16

SFS [D: (2,1) (2,1)] U/m SFS [D: (2,1) (2,1)], m = [ -5,13 | -2,5 ] : \#1: 7.236067977499755

Z\_4 + Z\_20

SFS [D: (2,1) (2,1)] U/m SFS [D: (2,1) (3,1)], m = [ -5,8 | -2,3 ] : \#1: 7.788854381999799

Z\_2 + Z\_18 \textcolor{red}{$\bigstar$}

SFS [D: (2,1) (2,1)] U/m SFS [D: (2,1) (3,2)], m = [ -5,8 | -2,3 ] : \#1: 7.236067977499767

Z\_2 + Z\_30

SFS [D: (2,1) (2,1)] U/m SFS [D: (2,1) (2,1) (2,1)], m = [ 1,1 | -1,0 ] : \#1: 17.950155281000672

2 Z\_2 + Z\_32

SFS [D: (2,1) (2,1)] U/m SFS [D: (2,1) (2,1) (2,1)], m = [ 1,1 | 0,-1 ] : \#1: 11.478019326001158

2 Z\_2 + Z\_32

SFS [D: (2,1) (2,1)] U/m SFS [D: (2,1) (2,1) (2,1)], m = [ 2,1 | 1,0 ] : \#1: 17.95015528100069

2 Z\_2 + Z\_28

SFS [D: (2,1) (2,1)] U/m SFS [D: (2,1) (2,1) (2,1)], m = [ 2,1 | 1,1 ] : \#1: 15.05572809000081

2 Z\_2 + Z\_20

SFS [D: (2,1) (2,1)] U/m SFS [D: (2,1) (2,1) (2,1)], m = [ -1,3 | 0,1 ] : \#1: 34.04984471899907

2 Z\_2 + Z\_16

SFS [D: (2,1) (2,1)] U/m SFS [D: (2,1) (2,1) (2,1)], m = [ 2,1 | 3,1 ] : \#1: 12.1613008990009

2 Z\_2 + Z\_4

SFS [D: (2,1) (2,1)] U/m SFS [D: (2,1) (2,1) (2,1)], m = [ -2,3 | -1,1 ] : \#1: 27.83869910099898

2 Z\_2 + Z\_12

SFS [D: (2,1) (2,1)] U/m SFS [D: (2,1) (2,1) (2,1)], m = [ -2,3 | -1,2 ] : \#1: 28.521980673998794

2 Z\_2 + Z\_4

SFS [D: (2,1) (2,1)] U/m SFS [D: (2,1) (2,1) (2,1)], m = [ -1,3 | -1,2 ] : \#1: 36.52198067399852

2 Z\_2 + Z\_16

SFS [D: (2,1) (2,1)] U/m SFS [D: (2,1) (2,1) (2,1)], m = [ -2,3 | -3,4 ] : \#1: 22.049844718999115

2 Z\_2 + Z\_4

SFS [D: (2,1) (2,1)] U/m SFS [D: (2,1) (2,1) (2,1)], m = [ -1,3 | -1,4 ] : \#1: 32.944271909999095

Z + 2 Z\_2

SFS [D: (2,1) (2,1)] U/m SFS [D: (2,1) (2,1) (3,1)], m = [ 1,1 | 0,1 ] : \#1: 7.527864045000402

Z\_4 + Z\_20 \textcolor{red}{$\bigstar$}

SFS [D: (2,1) (2,1)] U/m SFS [D: (2,1) (2,1) (3,1)], m = [ 1,1 | 1,0 ] : \#1: 7.527864045000403

Z\_4 + Z\_20 \textcolor{red}{$\bigstar$}

SFS [D: (2,1) (2,1)] U/m SFS [D: (2,1) (2,1) (3,1)], m = [ 1,1 | 1,2 ] : \#1: 8.080650449500455

2 Z\_4 \textcolor{red}{$\bigstar$}

SFS [D: (2,1) (2,1)] U/m SFS [D: (2,1) (2,1) (3,2)], m = [ 1,1 | 0,1 ] : \#1: 8.975077640500356

Z\_4 + Z\_28 \textcolor{red}{$\bigstar$}

SFS [D: (2,1) (2,1)] U/m SFS [M/n2: (2,1)], m = [ -2,-1 | 1,1 ] : \#1: 8.552786404500011

Z\_4 + Z\_24 \textcolor{red}{$\bigstar$}

SFS [D: (2,1) (2,1)] U/m SFS [M/n2: (2,1)], m = [ -1,-1 | 2,1 ] : \#1: 8.552786404500011

2 Z\_8 \textcolor{red}{$\bigstar$}

SFS [D: (2,1) (2,1)] U/m SFS [M/n2: (2,1)], m = [ -1,3 | 0,1 ] : \#1: 7.788854381999817

2 Z\_8 \textcolor{red}{$\bigstar$}

SFS [D: (2,1) (2,1)] U/m SFS [M/n2: (2,1)], m = [ 1,-3 | 0,1 ] : \#1: 11.788854381999764

2 Z\_8 \textcolor{red}{$\bigstar$}

SFS [D: (2,1) (2,1)] U/m SFS [M/n2: (2,1)], m = [ -2,3 | 1,-2 ] : \#1: 9.91934955049948

Z\_4 + Z\_8 \textcolor{red}{$\bigstar$}

SFS [D: (2,1) (2,1)] U/m SFS [M/n2: (2,1)], m = [ -2,3 | 1,-1 ] : \#1: 10.683281572999679

Z\_4 + Z\_8 \textcolor{red}{$\bigstar$}

SFS [D: (2,1) (2,1)] U/m SFS [M/n2: (2,1)], m = [ -1,3 | 1,-2 ] : \#1: 9.024922359499575

2 Z\_8 \textcolor{red}{$\bigstar$}

SFS [D: (2,1) (2,1)] U/m SFS [M/n2: (3,1)], m = [ 1,1 | 0,-1 ] : \#1: 10.633436854000454

2 Z\_2 + Z\_24

SFS [D: (2,1) (2,1)] U/m SFS [M/n2: (3,1)], m = [ 1,-2 | 0,1 ] : \#1: 10.683281572999697

Z\_4 + Z\_12 \textcolor{red}{$\bigstar$}

SFS [D: (2,1) (2,1)] U/m SFS [M/n2: (3,1)], m = [ -1,2 | 1,-1 ] : \#1: 10.683281572999693

Z\_4 + Z\_12 \textcolor{red}{$\bigstar$}

SFS [D: (2,1) (2,1)] U/m SFS [M/n2: (3,2)], m = [ 0,1 | 1,-2 ] : \#1: 8.422291236000298

2 Z\_2 + Z\_12

SFS [D: (2,1) (2,1)] U/m SFS [M/n2: (3,2)], m = [ -1,2 | 1,-1 ] : \#1: 9.919349550499494

Z\_4 + Z\_12 \textcolor{red}{$\bigstar$}

SFS [D: (2,1) (2,1)] U/m SFS [M/n2: (5,3)], m = [ 0,1 | 1,-1 ] : \#1: 9.527864045000399

2 Z\_2 + Z\_20

SFS [D: (2,1) (3,1)] U/m SFS [D: (2,1) (2,1) (2,1)], m = [ 0,-1 | 1,0 ] : \#1: 11.447213595499914

Z\_2 + Z\_54 \textcolor{red}{$\bigstar$}

SFS [D: (2,1) (3,1)] U/m SFS [D: (2,1) (2,1) (2,1)], m = [ 0,1 | 1,-1 ] : \#1: 8.211145618000156

Z\_2 + Z\_26 \textcolor{red}{$\bigstar$}

SFS [D: (2,1) (3,2)] U/m SFS [D: (2,1) (2,1) (2,1)], m = [ 0,-1 | 1,0 ] : \#1: 8.975077640500354

Z\_2 + Z\_66 \textcolor{red}{$\bigstar$}

SFS [D: (2,1) (3,2)] U/m SFS [D: (2,1) (2,1) (2,1)], m = [ 1,1 | 0,1 ] : \#1: 7.527864045000412

Z\_2 + Z\_50 \textcolor{red}{$\bigstar$}

SFS [D: (2,1) (3,2)] U/m SFS [D: (2,1) (2,1) (2,1)], m = [ 1,1 | 1,0 ] : \#1: 8.975077640500354

Z\_2 + Z\_54 \textcolor{red}{$\bigstar$}

SFS [D: (3,1) (3,1)] U/m SFS [D: (2,1) (2,1) (2,1)], m = [ -1,1 | 1,0 ] : \#1: 8.975077640500354

Z\_6 + Z\_18 \textcolor{red}{$\bigstar$}

SFS [D: (3,2) (3,2)] U/m SFS [D: (2,1) (2,1) (2,1)], m = [ -1,1 | 0,1 ] : \#1: 7.527864045000394

Z\_2 + Z\_30 \textcolor{red}{$\bigstar$}

SFS [D: (3,2) (3,2)] U/m SFS [D: (2,1) (2,1) (2,1)], m = [ -1,1 | 1,0 ] : \#1: 8.080650449500455

2 Z\_6 \textcolor{red}{$\bigstar$}

SFS [D: (2,1) (2,1)] U/m SFS [A: (4,1)] U/n SFS [D: (2,1) (2,1)], m = [ 0,1 | 1,0 ], n = [ 1,1 | 1,0 ] : \#1: 7.447213595499939

Z\_4 + Z\_12 \textcolor{red}{$\bigstar$}

SFS [D: (2,1) (2,1)] U/m SFS [A: (4,3)] U/n SFS [D: (2,1) (2,1)], m = [ 0,1 | 1,0 ], n = [ 1,1 | 1,0 ] : \#1: 7.788854381999809

2 Z\_4 \textcolor{red}{$\bigstar$}

SFS [D: (2,1) (2,1)] U/m SFS [A: (5,2)] U/n SFS [D: (2,1) (2,1)], m = [ 0,1 | 1,0 ], n = [ 1,1 | 1,0 ] : \#1: 8.21114561800015

Z\_4 + Z\_12 \textcolor{red}{$\bigstar$}

SFS [D: (2,1) (2,1)] U/m SFS [A: (5,3)] U/n SFS [D: (2,1) (2,1)], m = [ 0,1 | 1,0 ], n = [ 0,1 | 1,0 ] : \#1: 8.21114561800015

Z\_4 + Z\_28 \textcolor{red}{$\bigstar$}

SFS [D: (2,1) (2,1)] U/m SFS [A: (5,3)] U/n SFS [D: (2,1) (2,1)], m = [ 0,1 | 1,0 ], n = [ 1,1 | 1,0 ] : \#1: 16.42229123600025

2 Z\_2 + Z\_8

SFS [D: (2,1) (2,1)] U/m SFS [A: (5,3)] U/n SFS [D: (2,1) (2,1)], m = [ 0,-1 | 1,0 ], n = [ 0,1 | 1,0 ] : \#1: 8.21114561800015

Z\_4 + Z\_12 \textcolor{red}{$\bigstar$}

\subsubsection{Analysis}

Here  \textcolor{red}{$\bigstar$} means the Heegaard genus is larger than the rank of $H_1$.   These are potential counterexamples.

\begin{prop}{(Proposition 1.5 in \cite{BZ})}
\[
G= \langle s_{1}, \ldots, s_{2 q}, f \mid s_{1}^{2} f, \ldots, s_{2 q-1}^{2} f, s_{2 q}^{2 \lambda+1} f^{\beta}, s_{1} \ldots s_{2 q} f^{e} \rangle
\]
with  $q \geq 2, \lambda>0$  and  $\beta $ arbitrary, admits a presentation with  $2 q-2$  generators and  $2 q-2$  defining relators. 
\end{prop}

Usually the fundamental group considered is a quotient of $G$. We write \textcolor{blue}{True by \cite{BZ}. } if the manifold is a counterexample  by this proposition.

Observe that among the $92$ potential counterexamples, $35$ of them are marked as above. By Theorem 3.1 in \cite{BZ}, they are counterexamples. Also, in the list, any manifold that can be judged by Theorem 3.1 are marked.

Note that the Seifert Euler number in \cite{BZ} differs from the one defined in \cite{martelli} by a sign. 
%

Theorem 2.1 and Theorem 1.1 in \cite{BZ} also help, and $17$ manifolds are marked with  \textcolor{blue}{False. }


\subsubsection{Other remarks}

In \cite{weidmann} the writer describes a family of $3$-dimensional graph manifolds consisting of two
Seifert fibre spaces whose base orbifolds are the M\"obius band and a disk and which have one, 
repectively two, exceptional fibres. The considered spaces have $2$-generated fundamental
group but are not of Heegaard genus $2$. It is not difficult to see that these manifolds are
of genus $3$. The precise description of those manifolds is given in Theorem 1 there. 
%
%
%
%
%
%
%
%

Observe that as manifolds, SFS [D: (2,1) (2,-1)]$\cong$  the circle bundle over the M\"obius strip with total space orientable. (See Proposition 10.3.33 in \cite{martelli}. )
As  Seifert fibrations, SFS [D: (2,k) (2,j)]$\cong$SFS [D: (2,1) (2,-1)] for $k,j$ odd, and they need to be considered. (See Proposition 10.3.11 in \cite{martelli}. ) 
The TQFT method only tells us that among the $92$ selected manifolds, genus $\geq 3$. What we can detect is the case genus $\geq 3$ and rank $=2$.

Accordingly, the author thinks that the following manifold might be another  counterexample: 

SFS [D: (2,1) (2,1)] U/m SFS [D: (2,1) (3,1)], m = [ -5,8 | -2,3 ] : \#1: 7.788854381999799

Z\_2 + Z\_18  
\\

The case $r=4$ has been considered by the author and nothing interesting was found in the database.

There is also a  census of smallest known closed hyperbolic $3$-manifolds (where ``smallest" refers to volume), as tabulated by Hodgson and Weeks \cite{weeks}. 
The cases $r=3,4,5,6$ have been considered by the author and nothing interesting was found in this database.


\section*{Acknowledgements}

I would like to thank Professor Nicolai Reshetikhin for guiding me to quantum topology, which is the key tool in this thesis, and is in itself interesting and useful. 
Also I would like to thank Professor Yi Liu for introducing me to the conjecture, and Professor Xiaoming Du for discussions before the use of TQFT. I thank Professor Yin Tian for telling me there is something called TQFT. 

I thank Professor Qingtao Chen, Tian Yang and Tao Li for related talks, and Professor Yang Su for inviting me to give a talk at Academy of Mathematics and Systems Science, CAS. I thank Xiang Liu for his idea after the talk, and Professor Yinghua Ai, Si Li,  Zhongzi Wang and Fan Ye for discussions. 

I thank teachers in Yau Mathematical Sciences Center and Department of Mathematical Sciences. Finally I would like to thank my family, especially Rui Ling, for companion.


\section*{About \textit{Regina}}
%
%
%

\textit{Regina} is a software for low-dimensional topology, with a focus on triangulations, knots and links, normal surfaces, and angle structures. 

Any database used here can be downloaded from: 

\href{https://regina-normal.github.io/data.html}{\texttt{https://regina-normal.github.io/data.html}}.

\nocite{*}
\bibliographystyle{amsalpha} 
\bibliography{Bibliography} 

\end{document}